\documentclass[11pt]{article}
\usepackage{amsmath}
\usepackage{amssymb}
\usepackage{amsfonts}
\usepackage{mathrsfs}
\usepackage{cite}
\usepackage{cleveref}
\newtheorem{theorem}{Theorem}[section]

\newtheorem{lemma}{Lemma}[section]

\numberwithin{equation}{section}

\newenvironment{proof}{{\noindent \bf Proof.}}{\hfill $\Box$}
\newenvironment{proof3.1}{{\noindent \bf Proof of Theorem 3.1.}}{\hfill $\Box$}
\allowdisplaybreaks[4]

\begin{document}

\setlength{\baselineskip}{16pt}{\setlength\arraycolsep{2pt}}

\title{Stability of piezoelectric beams with magnetic effects of fractional derivative type and with/without thermal effects}

\author{Yanning An,\ \ Wenjun Liu\footnote{Corresponding author \ \   Email address: wjliu@nuist.edu.cn (W. J. Liu).}\medskip \ \ and Aowen Kong \\School of Mathematics and Statistics,
Nanjing University of Information Science \\
 and Technology, Nanjing 210044, China}


\date{}
\maketitle

\begin{abstract}
In this paper, we first consider the well-posedness and asymptotic behavior of a  one-dimensional piezoelectric beam system with control boundary conditions of fractional derivative type, which represent magnetic effects on the system. By introducing two new equations to deal with control boundary conditions of fractional derivative type, we obtain a new equivalent system. With the help of Lumer-Philips theorem, we also show the well-posedness of the system. We then prove the lack of exponential stability by a spectral analysis, and obtain the polynomial stability of the system by using a result of Borichev and Tomilov (Math. Ann. {\bf 347} (2010), 455--478). To find a more stable system, we then consider the stability of the above system  with additional thermal effects described by Fourier's law, and achieve the exponential stability for the new model by using the perturbed functional method.
\end{abstract}


\maketitle

\section{Introduction }

Piezoelectric materials have the inherent characteristic of piezoelectricity. Piezoelectric materials with asymmetric crystal structure can generate electric energy when they are mechanically compressed, resulting in piezoelectric effect. In recent years, piezoelectric materials have been used in the design of intelligent devices \cite{dag2014,guzhu2016,giu2008,ertinm2008}. There are two types of piezoelectric effect. The first is called direct piezoelectric effect, which is characterized by the internal generation of electric charge, by applying mechanical force \cite{ka1969,publi1996}. In the second case, it is called the inverse piezoelectric effect, the internal mechanical stress comes from the excitation generated by the external electric field.
 Since the above two effects are caused by the asymmetry of crystal structure, direct piezoelectric effect and reverse piezoelectric effect have the same working principle \cite{galdin2000}. That is, piezoelectric materials can convert mechanical motion and electric field. As we all know, machinery, electricity and magnetism can transform and influence each other. When piezoelectric materials are integrated into the components of electronic circuits, the piezoelectric materials driven by voltage, current or charge to the electrodes will produce mechanical motion and magnetic field affecting the stability of the system. Therefore, it is necessary to describe the influence of the interaction of these three effects (mechanical, electrical and magnetic) on the stability of piezoelectric system \cite{bsw1996,dwg2010,yang2006}. The equation of piezoelectric beam with magnetic effects is based on the mechanical behavior of beam by Mindlin-Timoshenko theory {\cite{bsw1996,dwg2010}} and the description of electromagnetic coupling by Maxwell equation.

The first work to consider the three effects (mechanical, electrical and magnetic) is due to Morris and \"{O}zer \cite{mo2013,mo2014}. They studied a piezoelectric system with only one boundary control, showed by
\begin{equation} \label{other1.2}
\begin{aligned}
&\rho v_{tt}-\alpha v_{xx}+\gamma\beta p_{xx}=0,&(x,t)\in(0,L)\times (0,T),\\
&\mu p_{tt}-\beta p_{xx}+\gamma\beta v_{xx}=0,&(x,t)\in(0,L)\times (0,T),
\end{aligned}
\end{equation}
with boundary conditions
\begin{equation} \label{other1.2b}
\begin{aligned}
&v(0,t)=\alpha v_{x}(L,t)-\gamma\beta p_{x}(L,t)=0,& t\in (0,T),\\
&p(0,t)=\beta p_{x}(L,t)-\gamma\beta v_{x}(L,t)+V(t)=0, & t\in (0,T),
\end{aligned}
\end{equation}
where $\rho$, $\alpha$, $\beta$, $\gamma$, $\mu$  denote the mass density per unit volume, the elastic stiffness, the beam coefficient of impermeability, the piezoelectric coefficient and the magnetic permeability, respectively; And $v(x,t)$, $p(x,t)$ and $V(t)=\frac{p_{t}(L,t)}{h}$ represent the transverse displacement of the beam and the total load of the electric displacement along the transverse direction at each point $x$ and the prescribed voltage on the beam electrodes respectively. Let $\alpha_{1}>0$ be the elastic stiffness of the model derived from the electrostatic and quasi-static methods of Euler Bernoulli small displacement (see \cite{mo2014}), and $\alpha,\beta$ and $\gamma$ satisfies $\alpha=\alpha_{1}+\gamma^{2}\beta$.
The authors showed that system \eqref{other1.2}-\eqref{other1.2b}  was not exponentially stable. At the same time, they also obtained an exact observability inequality on the boundary.

In \cite{rfajm2019}, Ramos et al. studied the system of piezoelectric beams with magnetic effects, showed by
\begin{equation} \label{other1.3}
\begin{aligned}
&\rho v_{tt}-\alpha v_{xx}+\gamma\beta p_{xx}=0,&(x,t)\in(0,L)\times (0,T),\\
&\mu p_{tt}-\beta p_{xx}+\gamma\beta v_{xx}=0,&(x,t)\in(0,L)\times (0,T),
\end{aligned}
\end{equation}
with boundary conditions
\begin{equation} \label{other1.3b}
\begin{aligned}
&v(0,t)=\alpha v_{x}(L,t)-\gamma\beta p_{x}(L,t)+\frac{\xi_{1}}{h}v_{t}(L,t)=0,& t\in (0,T),\\
&p(0,t)=\beta p_{x}(L,t)-\gamma\beta v_{x}(L,t)+\frac{\xi_{2}}{h}p_{t}(L,t)=0,& t\in (0,T),
\end{aligned}
\end{equation}
where $\xi_{i},i=1,2$ are positive constant feedback gains. The authors first divided the original system into a conservative system and an auxiliary system.
And they proved the observability inequality of conservative system on the boundary by using multiplier method.
Then, the exponential stability of system \eqref{other1.3}-\eqref{other1.3b} was provided by demonstrating the equivalence between observability inequality of conservative system and exponential stability of original system.

Recently, some researchers have studied Timoshenko system or hybrid system with control boundary conditions of fractional derivative type, see \cite{BB2017, BG2019, mary2019,tita2018,RPV2021}.
The fractional derivative in time direction can be regarded as the dissipation term of the system and make the solution converge to the equilibrium state, see \cite{bou2019}. Therefore, when being applied to the boundary, it can be regard as a controller that helps suppress unwanted vibrations.
Since control of fractional order type can lead to a more robust control performance and more adequate modeling, it plays an important role in practical application. At present, it has been widely used in bioengineering, circuit, signal processing, chemical process, control system and viscoelasticity \cite{pk2011}.
For example, in viscoelastic system, because the properties of viscoelastic materials include the properties of both elastic solids and viscous fluids, the stress at each point and at each moment in the motion of viscoelastic objects depends not only on the current value of strain, but also on the whole temporal prehistory of motion from $0$ to time $t$.
In this paper, fractional dissipation describes a magnetic effect controller, which aims to provide boundary dissipation to stabilize the system.

Since fractional calculus was first proposed in 1969, many authors have given a variety of definitions.
In this paper, we consider the exponential Caputo's fractional
derivative $\partial^{a,\eta}_{t}$ of order $a\in(0,1)$ with respect to time variable $t$ defined by
\begin{align*}
\partial^{a,\eta}_{t}f(t)=\frac{1}{\Gamma(1-a)}\int^{t}_{0}e^{-\eta(t-\tau)}(t-\tau)^{-a}\frac{d}{d\tau}f(\tau)d\tau,
\end{align*}
with $f\in L^{1}(I)$ and $t>0$. The fractional differentiation $I^{a,\eta}$ is inverse operation of fractional integration defined by
\begin{align*}
I^{a,\eta}f(t)=\frac{1}{\Gamma(a)}\int^{t}_{0}e^{-\eta(t-\tau)}(t-\tau)^{a-1}f(\tau)d\tau,
\end{align*}
with $f\in W^{1,1}(I)$ and $t>0$, and we know that $\partial^{a,\eta}_{t}f(t)=I^{1-a,\eta}f'(t)$.

In the first part of this paper, we consider the asymptotic behavior of the piezoelectric system without considering temperature. The system is showed by
\begin{equation} \label{problem1.1}
\begin{aligned}
&\rho V_{tt}-\alpha V_{xx}+\gamma\beta P_{xx}=0, &(x,t)\in(0,L)\times (0,T),\\
&\mu P_{tt}-\beta P_{xx}+\gamma\beta V_{xx}=0,  &(x,t)\in(0,L)\times (0,T),
\end{aligned}
\end{equation}
with the double boundary conditions
\begin{equation} \label{1.2}
\begin{aligned}
&V(0,t)=P(0,t)=0,& t\in (0,T),\\
&\alpha V_{x}(L,t)-\gamma\beta P_{x}(L,t)=-l_{1}\partial^{a,\eta}_{t}V(L,t),& t\in (0,T),\\
&\beta P_{x}(L,t)-\gamma\beta V_{x}(L,t)=-l_{2}\partial^{a,\eta}_{t}P(L,t),& t\in (0,T),
\end{aligned}
\end{equation}
and the initial conditions
\begin{align}\label{1.3}
\left(V(x,0),V_{t}(x,0),P(x,0),P_{t}(x,0)\right)=\left(V_{0}(x),V_{1}(x),P_{0}(x),P_{1}(x)\right),x\in(0,L).
\end{align}
where $V(x,t), P(x,t)$  represent the transverse displacement of the beam and the total load of the electric displacement along the transverse direction at each point $x$ respectively, $l_{1},l_{2}>0$ are positive constants feedback gains, $\eta>0,a\in(0,1)$ are constants. We shall study the polynomial stability of this piezoelectric beam system with magnetic effects of fractional derivative type.

In the above problem, we find that the piezoelectric beam system with control boundary conditions of fractional derivative type is not exponentially stable. In order to find a more stable system, we then consider the stability of the system \eqref{problem1.1}-\eqref{1.3} with thermal effects described by Fourier's law. That is, when the model is subject to the temperature variation, we consider the thermal effect acting on the transverse displacement equation of the beam, and  propose a new piezoelectric beam system given by
\begin{equation} \label{problem1.2}
\begin{aligned}
&\rho V_{tt}-\alpha V_{xx}+\gamma\beta P_{xx}+\delta\theta_{x}=0, &(x,t)\in(0,L)\times (0,T),\\
&\mu P_{tt}-\beta P_{xx}+\gamma\beta V_{xx}=0,  &(x,t)\in(0,L)\times (0,T),\\
&c\theta_{x}-k\theta_{xx}+\delta V_{xt}=0, &(x,t)\in(0,L)\times (0,T),
\end{aligned}
\end{equation}
with the double boundary conditions
\begin{equation} \label{1.5}
\begin{aligned}
&V(0,t)=P(0,t)=\theta_x(0,t)=\theta(L,t)=0,& t\in (0,T),\\
&\alpha V_{x}(L,t)-\gamma\beta P_{x}(L,t)=-l_{1}\partial^{a,\eta}_{t}V(L,t),& t\in (0,T),\\
&\beta P_{x}(L,t)-\gamma\beta V_{x}(L,t)=-l_{2}\partial^{a,\eta}_{t}P(L,t),& t\in (0,T),
\end{aligned}
\end{equation}
and the initial conditions
\begin{align}\label{1.6}
\left(V(x,0),V_{t}(x,0),P(x,0),P_{t}(x,0),\theta(x,0)\right)=\left(V_{0}(x),V_{1}(x),P_{0}(x),P_{1}(x),\theta_0(x)\right),x\in(0,L),
\end{align}
where $\theta(x, t)$ is the temperature (difference to a fixed constant reference temperature), $\kappa$ is the thermal conductivity, $c$ is the specific heat and $\delta$ is the product of the thermal expansion and elastic moduli.

The difficulty of these two problems lies in the treatment of boundary fractional damping. To solve this problem, we will first introduce two new functions to transform the system into an equivalent augmented system. Then we will analyze the well-posedness of the system by using the Lumer-Philips theorem. To deal with the stability analysis, we divide system \eqref{problem1.1}-\eqref{1.3} into a conservative system and a zero initial value system. Then, the polynomial stability of the system will be given by using a result of Borichev and Tomilov. Many scholars have used different methods to analyze the polynomial stability of some other systems, see \cite{liuzhao2019,bm2020,HR2020,kam2016,maycmg2019,tita2018,lkl2020,am2020,must2021,card2019,feng2017}.
Finally, for the asymptotic behavior analysis part of system \eqref{problem1.2}-\eqref{1.6}, we obtain the exponential stability of the system by using perturbed functional method.

The structure of this paper is as follows. In the next section, we will give the well-posedness of system \eqref{problem1.1}-\eqref{1.3}. In Section 3, we will show the lack of exponential stability. Next, we will obtain the polynomial stability of system \eqref{problem1.1}-\eqref{1.3} in Section 4. In Section 5, we will consider the well-posedness of  system \eqref{problem1.2}-\eqref{1.6}. Finally, we will show that system \eqref{problem1.2}-\eqref{1.6} is exponential stability in Section 6.

\section{Well-posedness for piezoelectric beams without thermal effects}

In this section, by using a semigroup approach, we establish well-posedness result for the piezoelectric beam \eqref{problem1.1} with control boundary conditions \eqref{1.2}. First, in order to transform the boundary conditions of fractional derivative type, we introduce two new functions $\phi_1, \phi_2$ which can plant model \eqref{problem1.1}-\eqref{1.3} into an augmented system. Therefore, we need the following theorem.
\begin{theorem}\label{thm2.1}
\textup{(\cite{mbo2006})}
Let $\mu$ be the function
\begin{align*}
\mu(\xi)=|\xi|^{(2a-1)/2},\quad -\infty<\xi<+\infty,\quad 0<a<1.
\end{align*}
Then the relation between the ``input'' $U_i,i=1,2$ and the ``output'' $O_i,i=1,2$ of system
\begin{align*}
\partial_{t}\phi_i(\xi,t)+\left(\xi^{2}+\eta\right)\phi_i(\xi,t)-U_i(t)\mu(\xi)&=0,\quad-\infty<\xi<+\infty,\quad\eta\geq0,\quad t>0,\\
\phi_i(\xi,0)&=0,\quad-\infty<\xi<+\infty,\\
O_i(t)&=\frac{\sin(a\pi)}{\pi}\int^{+\infty}_{-\infty}\mu(\xi)\phi_i(\xi,t)d\xi,
\end{align*}
is given by
\begin{align*}
O_i(t)=I^{1-a,\eta}U_i(t).
\end{align*}
\end{theorem}

Then, taking $U_1(t)=V_t(L,t)$ and $U_2(t)=P_t(L,t)$, using Theorem \ref{thm2.1} and the fact that $\partial^{a,\eta}_t f(t)=I^{1-a,\eta}f'(t)$, we can obtain
\begin{align}\label{2.1}
\partial^{a,\eta}_tV(L,t)=I^{1-a,\eta}V_t(L,t)=I^{1-a,\eta}U_1=O_1(t)=\frac{\sin(a\pi)}{\pi}\int^{+\infty}_{-\infty}\mu(\xi)\phi_1(\xi,t)d\xi,
\end{align}
and
\begin{align}\label{2.101}
\partial^{a,\eta}_tP(L,t)=I^{1-a,\eta}P_t(L,t)=I^{1-a,\eta}U_2=O_2(t)=\frac{\sin(a\pi)}{\pi}\int^{+\infty}_{-\infty}\mu(\xi)\phi_2(\xi,t)d\xi.
\end{align}

In \cite{mbo2006}, there exists a detailed proof about Theorem \ref{thm2.1}. The following lemmas are the conclusions that we will use in the flow-up process of our proof.
\begin{lemma}\label{nuse1}
\textup{(\cite{mbo2006})} If $\lambda\in E$, where $E=\{\lambda\in C: Re\lambda+\eta>0\}\cup\{\lambda\in{C}:Im\lambda\neq0\}$. Then
\begin{align*}
\int^{+\infty}_{-\infty}\frac{\mu^{2}(\xi)}{\xi^{2}+\eta+|\lambda|}d\xi=\frac{\pi}{\sin(a\pi)}(\eta+|\lambda|)^{a-1}.
\end{align*}
\end{lemma}
\begin{lemma}\label{nuse3}
\textup{(\cite{akil2020})}
Let $0<a<1$, $\eta>0$, and $\lambda>0$, then
\begin{equation*}
\left\{
\begin{aligned}
&\int^{+\infty}_{-\infty}\frac{|\xi|^{a+\frac{1}{2}}}{(\lambda+\xi^2+\eta)^2}d\xi=c_1(\lambda+\eta)^{\frac{a}{2}-\frac{5}{4}},\\
&\left(\int^{+\infty}_{-\infty}\frac{1}{(\lambda+\xi^2+\eta)^2}d\xi\right)^{\frac{1}{2}}=\sqrt{\frac{\pi}{2}}\frac{1}{(\lambda+\eta)^{\frac{3}{4}}},\\
&\left(\int^{+\infty}_{-\infty}\frac{\xi^2}{(\lambda+\xi^2+\eta)^4}d\xi\right)^{\frac{1}{2}}=\frac{\sqrt{\pi}}{4}\frac{1}{(\lambda+\eta)^{\frac{5}{4}}}
\end{aligned}
\right.
\end{equation*}
where $c_1$ is a positive constant number independent of $\lambda$.
\end{lemma}

Using equalities \eqref{2.1} and \eqref{2.101}, system \eqref{problem1.1}-\eqref{1.3} can be rewritten as the augmented model
\begin{align}
&\rho V_{tt}-\alpha V_{xx}+\gamma\beta P_{xx}=0,&&(x,t)\in  (0,L)\times (0,T), \label{2.3}\\
&\partial_{t}\phi_{1}(\xi,t)+\left(\xi^{2}+\eta\right)\phi_{1}(\xi,t)-V_{t}(L,t)\mu(\xi)=0,&&(\xi,t)\in(-\infty,+\infty)\times(0,+\infty),\label{2.301}\\
&\mu P_{tt}-\beta P_{xx}+\gamma\beta V_{xx}=0,  && (x,t)\in(0,L)\times (0,T),\label{2.4}\\
&\partial_{t}\phi_{2}(\xi,t)+\left(\xi^{2}+\eta\right)\phi_{2}(\xi,t)-P_{t}(L,t)\mu(\xi)=0,&&(\xi,t)\in(-\infty,+\infty)\times(0,+\infty),\label{2.401}
\end{align}
with the boundary conditions
\begin{align}
&V(0,t)=P(0,t)=0,& t\in (0,T),\\
&\alpha V_{x}(L,t)-\gamma\beta P_{x}(L,t)=-l_{1}\frac{\sin(a\pi)}{\pi}\int^{\infty}_{-\infty}\mu(\xi)\phi_{1}(\xi,t)d\xi,& t\in (0,T),\label{2.501}\\
&\beta P_{x}(L,t)-\gamma\beta V_{x}(L,t)=-l_{2}\frac{\sin(a\pi)}{\pi}\int^{\infty}_{-\infty}\mu(\xi)\phi_{2}(\xi,t)d\xi,& t\in (0,T)\label{2.502}
\end{align}
and the initial conditions
\begin{align}\label{2.7}
\left(V(x,0), V_{t}(x,0), \phi_{1}(0), P(x,0), P_{t}(x,0),\phi_{2}(0)\right)=\left(V_{0},V_{1},\phi_{01},P_{0},P_{1},\phi_{02}\right)\quad x \in (0,L).
\end{align}

The energy of system \eqref{2.3}-\eqref{2.7} is given by
\begin{align}\label{2.8}
E(t)=\frac{1}{2}\int^{L}_{0}\left[\rho\left|V_{t}\right|^{2}+\alpha_{1}\left|V_{x}\right|^{2}+\mu\left|P_{t}\right|^{2}+\beta\left|\gamma V_{x}-P_{x}\right|^{2}\right] dx+\frac{\sin(a\pi)}{2\pi}\int^{\infty}_{-\infty}\left(l_{1}|\phi_{1}|^{2}+l_{2}|\phi_{2}|^{2}\right)d\xi.
\end{align}
Multiplying \eqref{2.3}, \eqref{2.4} by $ V_{t}$, $P_{t}$ respectively, integrating on $(0,L)$, and multiplying \eqref{2.301}, \eqref{2.401} by $l_{1}\frac{\sin(a\pi)}{\pi}\phi_{1}$, $l_{2}\frac{\sin(a\pi)}{\pi}\phi_{2}$ respectively, integrating on $\mathbb{R}$, we can get
\begin{align}\label{2.9}
\frac{d}{dt}E(t)=-\frac{\sin(a\pi)}{\pi}\int^{\infty}_{-\infty}(\xi^{2}+\eta)\left(l_{1}|\phi_{1}|^{2}+l_{2}|\phi_{2}|^{2}\right)d\xi.
\end{align}

Let the space $\mathcal{H}$ be
\begin{align*}
\mathcal{H}:=H^{1}_{*}(0,L)\times {L}^{2}(0,L)\times{L}^{2}(-\infty,+\infty)\times {{H}}^{1}_{*}(0,L)\times {{L}}^{2}(0,L)\times{L}^{2}(-\infty,+\infty),
\end{align*}
where $ {{H}}^{1}_{*}(0,L)=\left\{f\in{{H}}^{1}(0,L):f(0)=0\right\}$. And the inner product in the space $\mathcal{H}$ is defined by
\begin{align*}
\langle{U}_{1},{U}_{2}\rangle_{\mathcal{H}}=&\int^{L}_{0}\left[\rho f_{1}\overline{f}_{2}+\mu g_{1}\overline{g}_{2}+\alpha_{1}V_{1,x}\overline{V}_{2,x}+\beta(\gamma V_{1,x}-P_{1,x})\overline{(\gamma V_{2,x}-P_{2,x})}\right]dx \\
&+\frac{\sin(a\pi)}{\pi}\int^{+\infty}_{-\infty}\left(l_{1}\phi_{1,1}\overline{\phi}_{1,2}+l_{2}\phi_{2,1}\overline{\phi}_{2,2}\right)dx,
\end{align*}
where ${{U}}_{i}=\left(V_{i},f_{i}, \phi_{1,i}, P_{i}, g_{i},\phi_{2,i}\right)\in {H},i=1,2$.
Let the vector function ${{U}}=\left(V,f, \phi_{1}, P, g, \phi_{2}\right)^{T}$,
then system \eqref{2.3}-\eqref{2.7} can be equivalent to an abstract development equation, which is \begin{equation} \label{2.10}
\left\{
\begin{aligned}
&{{U}}_{t}={\mathcal{A}}{{U}}\\
&{{U}}(0)={{U}}_{0}
\end{aligned}
\right.
\end{equation}
where ${{U}}_{0}=\left(V_{0},V_{1}, \phi_{1,0}, P_{0}, P_{1}, \phi_{2,0}\right)^{T}$ and the operator $\mathcal{A}:\mathcal{D}(\mathcal{A})\subset \mathcal{H}\rightarrow \mathcal{H}$ is showed  by
\begin{align*}
\mathcal{A}\begin{bmatrix} V \\ f \\ \phi_{1} \\ P \\ g\\ \phi_{2}\end{bmatrix}=
\begin{bmatrix}
   f \\
  \frac{\alpha}{\rho}V_{xx} -\frac{\gamma\beta}{\rho}P_{xx} \\
  -\left(\xi^{2}+\eta\right)\phi_{1}(\xi,t)+f(L,t)\mu(\xi)\\
  g \\
  -\frac{\gamma\beta}{\mu}V_{xx} + \frac{\beta}{\mu}P_{xx} \\
  -\left(\xi^{2}+\eta\right)\phi_{2}(\xi,t)+g(L,t)\mu(\xi)
\end{bmatrix}
,
\end{align*}
where the domain of the operator $\mathcal{A}$ is given by
\begin{align*}
\mathcal{D}(\mathcal{A}):=\bigg\{&{{U}}\in \mathcal{H};  V,P \in {H}^{2}(0,L)\cap H^{1}_{*}(0,L), f,g\in H^{1}_{*}(0,L),|\xi|\phi_{1},|\xi|\phi_{2}\in L^{2}(-\infty,+\infty),\\
&-\left(\xi^{2}+\eta\right)\phi_{1}(\xi,t)+f(L,t)\mu(\xi),-\left(\xi^{2}+\eta\right)\phi_{2}(\xi,t)+g(L,t)\mu(\xi)\in L^{2}(-\infty,+\infty)\bigg\}.
\end{align*}
with ${{U}}=\left(V,f, \phi_{1}, P, g, \phi_{2}\right)$.

For showing that the operator $\mathcal{A}$ generates a ${C}_{0}-$semigroup $\{\mathcal{S}_{\mathcal{A}}(t)\}_{t\geq0}$ of contractions on the space $\mathcal{H}$, we first give a lemma to show $\mathcal{A}$ is dissipative.

\begin{lemma}\label{lem2.1}
The operator $\mathcal{A}$ is dissipative and satisfies that for any ${U}\in \mathcal{D(\mathcal{A})}$,
\begin{align}\label{2.11}
{Re}\langle  {\mathcal{A}}{U},{U}\rangle_{{H}}=-\frac{\sin(a\pi)}{\pi}\int^{\infty}_{-\infty}(\xi^{2}+\eta)\left(l_{1}|\phi_{1}|^{2}+l_{2}|\phi_{2}|^{2}\right)d\xi\leq 0.
\end{align}
\end{lemma}
\begin{proof}
For any ${U}\in \mathcal{D(\mathcal{A})}$, with the help of the inner product in $\mathcal{H}$, we obtain that the lemma is correct.
\end{proof}
\begin{theorem}\label{them2.2}
Let ${U}_{0}\in \mathcal{D(\mathcal{A})}$, there exists a unique solution ${U}(t)=\mathcal{S}_{\mathcal{A}}(t){U}_{0}$ of \eqref{2.10} such that
\begin{align*}
{U}\in{C}\left([0,\infty);\mathcal{D}(\mathcal{A})\right)\cap{C}^{1} \left([0,\infty); \mathcal{H}\right).
\end{align*}
\end{theorem}
\begin{proof}
It relies on the Lumer-Phillips theorem, which ensures that the operator $\mathcal{A}$ is the infinitesimal generator of a $C_{0}-$semigroup $\{\mathcal{S}_{\mathcal{A}}(t)\}_{t\geq0}$. From \eqref{2.11}, we have shown that the operator $\mathcal{A}$ is dissipative, closed and dense.
So we need to prove that $0\in \rho(\mathcal{A})$. For any ${F}=\left(F_{1},F_{2},F_{3},F_{4},F_{5},F_{6}\right)\in \mathcal{H}$, we will prove there exists a ${{U}}=\left(V,f, \phi_{1}, P, g, \phi_{2}\right)$ such that
\begin{align*}
-\mathcal{A}{U}={F}.
\end{align*}
Equivalently, we need to consider the existence of unique solution of the system
\begin{align}
-f&=F_{1},\label{2.121}\\
 -{\alpha}V_{xx}+{\gamma\beta}P_{xx}&={\rho}F_{2},\label{2.122}\\
\left(\xi^{2}+\eta\right)\phi_{1}(\xi,t)-f(L)\mu(\xi)&=F_{3},\label{2.123}\\
-g&=F_{4},\label{2.124}\\
-{\beta}P_{xx}+{\gamma\beta}V_{xx}&={\mu}F_{5},\label{2.125}\\
\left(\xi^{2}+\eta\right)\phi_{2}(\xi,t)-g(L)\mu(\xi)&=F_{6}.\label{2.126}
\end{align}
Thanks to \eqref{2.121} and \eqref{2.123}, it follows that $f,g\in H^{1}_{*}(0,L)$ and
\begin{align}\label{2.131}
f=-F_{1},\quad & \quad g=-F_{4}.
\end{align}
Then combining \eqref{2.123} with \eqref{2.126} and \eqref{2.131}, we have that
\begin{align}\label{2.132}
\phi_{1}=\frac{-F_{1}(L)\mu(\xi)+F_{3}}{\xi^{2}+\eta},\quad\quad\phi_{2}=\frac{-F_{4}(L)\mu(\xi)+F_{6}}{\xi^{2}+\eta}.
\end{align}
From Lemma \ref{nuse1}, we conclude that $\phi_{i}\in L^{2}(R)$. Then, we need to prove the existence and uniqueness of solution for system
\begin{equation}\label{2.141}
\left\{
\begin{aligned}
& {\alpha}V_{xx}-{\gamma\beta}P_{xx}=-\rho F_{2} ,\\
&\beta P_{xx}-\gamma\beta V_{xx}=-\mu F_{5} ,\\
&\alpha V_{x}(L,t)-\gamma\beta P_{x}(L,t)=-l_{1}\frac{\sin(a\pi)}{\pi}\int^{\infty}_{-\infty}\mu(\xi)\frac{-F_{1}(L)\mu(\xi)+F_{3}(\xi)}{\xi^{2}+\eta}d\xi,\\
&\beta P_{x}(L,t)-\gamma\beta V_{x}(L,t)=-l_{2}\frac{\sin(a\pi)}{\pi}\int^{\infty}_{-\infty}\mu(\xi)\frac{-F_{4}(L)\mu(\xi)+F_{6}(\xi)}{\xi^{2}+\eta}d\xi.
\end{aligned}
\right.
\end{equation}
By using the Lax-Milgram theorem, we know that there exists a solution $(V,P)\in {H}^{1}_{*}(0,L)\times {H}^{1}_{*}(0,L)$. Combining \eqref{2.131}, \eqref{2.132} and \eqref{2.141}, we obtain that $0\in \rho(\mathcal{A})$. Thus, the proof of the theorem is completed.
\end{proof}

\section{Lack of exponential stability for piezoelectric beams without thermal effects}
In this section, we use the classical method  developed by  Maryati et al.\cite{bng2000,tita2018} to show that the piezoelectric beam system \eqref{2.3}-\eqref{2.7} is not exponentially stable. To do that, we need the following theorem:
\begin{theorem}\label{thm3.1}
\textup{(\cite{tita2018})}
Let $\mathcal{H}$ be a Hilbert space and $\mathcal{H}_{0}$ be a closed subspace of $\mathcal{H}$. Assume that $\{\mathcal{S}(t)\}_{t\geq0}$ is a contraction semigroup defined on $\mathcal{H}_{0}$ and $\{\mathcal{S}_{0}(t)\}_{t\geq0}$ is a unitary group over $\mathcal{H}_{0}$. If the difference $\{\mathcal{S}(t)-\mathcal{S}_{0}(t)\}_{t\geq0}$ is a compact operator from $\mathcal{H}_{0}$ to $\mathcal{H}$, the semigroup $\{\mathcal{S}(t)\}_{t\geq0}$ is not exponentially stable.
\end{theorem}

In order to achieve the proof of lack of exponential stability, we introduce some functionals
\begin{align*}
\mathcal{J}(x,t)&=\frac{1}{2}\big({\rho}\left|V_{t}\right|^{2}+{\alpha_{1}}\left|V_{x}\right|^{2}+{\mu}\left|P_{t}\right|^{2}+{\beta}\left|\gamma V_{x}-P_{x}\right|^{2}\big),\\
J(t)&=\int^{L}_{0}\big(\rho q V_{t}V_{x}+\mu q P_{t}(P_{x}-\gamma V_{x})+\mu\gamma q P_{t} V_{x}\big)dx,
\end{align*}
where $q$ is an appropriate function of $x$, and $q\in C^1[0,L]$.

Moreover, we will give the inequality relation between these functionals and energy functional $E(t)$, which is a very useful lemma for us to prove the lack of exponential stability.
\begin{lemma}\label{lem3.2}
There exists a positive constant $C$ such that
\begin{align}
\left|\int^{t}_{0}\left[q(L)\mathcal{J}(L,\tau)-q(0)\mathcal{J}(0,\tau)\right]d\tau - \int^{t}_{0} \int^{L}_{0}q'(x)\mathcal{J}(x,\tau)dxd\tau\right|=\left|J(0)-J(t)\right|\leq C E(0).
\end{align}
\end{lemma}
\begin{proof}
Multiplying \eqref{2.3} by $qV_{x}$, we have
\begin{align}\label{3.21}
\rho qV_{x}V_{tt}-\alpha qV_{x} V_{xx}+\gamma\beta qV_{x}P_{xx}=0.
\end{align}
And it is straightforward to verify that
\begin{align}\label{3.22}
\rho qV_{tt}V_{x}=\rho\frac{\partial}{\partial t}\left(qV_{t}V_{x}\right)-\rho q V_{t}V_{xt}=\rho\frac{\partial}{\partial t}\left(qV_{t}V_{x}\right)- q \frac{\partial}{\partial x}\left(\frac{\rho}{2}|V_{t}|^{2}\right).
 \end{align}
By plugging \eqref{3.22} into \eqref{3.21}, we have
\begin{align}\label{3.23}
\rho\frac{\partial}{\partial t}\left(qV_{t}V_{x}\right)- q \frac{\partial}{\partial x}\left(\frac{\rho}{2}|V_{t}|^{2}+\frac{\alpha_{1}}{2}|V_{x}|^{2}\right)=\gamma\beta q\left(\gamma V_{xx}-P_{xx}\right)V_{x}.
\end{align}
Multiplying \eqref{2.4} by $qP_{x}$, and using similar techniques, we have
\begin{align}\label{3.24}
\mu\frac{\partial}{ \partial t}\left(qP_{t}P_{x}\right)- q \frac{\partial}{\partial x}\left(\frac{\mu}{2}|P_{t}|^{2}\right)=-\beta q\left(\gamma V_{xx}-P_{xx}\right)P_{x}.
\end{align}
Adding \eqref{3.23} and \eqref{3.24}, we can obtain
\begin{align*}
\frac{\partial}{\partial t}\left(\rho q V_{t}V_{x}+\mu q P_{t}(P_{x}-\gamma V_{x})+\mu\gamma q P_{t} V_{x}\right)=q\frac{\partial}{ \partial x}\left(\frac{\rho}{2}\left|V_{t}\right|^{2}+\frac{\alpha_{1}}{2}\left|V_{x}\right|^{2}+\frac{\mu}{2}\left|P_{t}\right|^{2}+\frac{\beta}{2}\left|\gamma V_{x}-P_{x}\right|^{2}\right).
\end{align*}
Let us integrate the above equality from $0$ to $L$, then
\begin{align*}
\frac{d}{dt}J(t)-\int^{L}_{0}q\frac{\partial}{\partial x}(\mathcal{J}(x,t))dx=0.
\end{align*}
Using partial integral for the second term, we can represent the above equation as
\begin{align*}
\frac{d}{dt}J(t)-\left[q(x)\mathcal{J}\right]|^{L}_{0}+\int^{L}_{0}q'(x)\mathcal{J}(x,t)dx=0.
\end{align*}
Integrating the above equality from $0$ to $t$, we have
\begin{align*}
[J(\tau)]|^{t}_{0}-\int^{t}_{0} \left[q(x)\mathcal{J}\right]|^{L}_{0}d\tau+\int^{t}_{0}\int^{L}_{0}q'(x)\mathcal{J}(x,t)dxd\tau=0.
\end{align*}
That is,
\begin{align*}
\left|\int^{t}_{0}\left[q(L)\mathcal{J}(L,\tau)-q(0)\mathcal{J}(0,\tau)\right]d\tau -\int^{t}_{0} \int^{L}_{0}q'(x)\mathcal{J}(x,\tau)dxd\tau\right|=\left|J(0)-J(t)\right|.
\end{align*}
We find that the terms $V_{x},V_{t},P_{t},\gamma V_{x}-P_{x}$ in $J$ are some terms in energy $E(t)$. If we use Young's inequality for $J$, we can easily get that
\begin{align*}
\left|J(0)-J(t)\right|\leq \left|J(0)\right|+\left|J(t)\right|\leq C E(t)+C E(0),
\end{align*}
where $C$ is a constant that only depends on $\alpha,\beta,\gamma,\rho,L,\|q\|_{\infty}$. From \eqref{2.9}, we have that the energy $E(t)$ is dissipative. Therefore
\begin{align*}
E(t)\leq E(0).
\end{align*}
The proof is completed.
\end{proof}

Now we consider the piezoelectric beam system without dissipation
\begin{align}
&\rho \widetilde{V}_{tt}-\alpha \widetilde{V}_{xx}+\gamma\beta \widetilde{P}_{xx}=0, &(x,t)\in(0,L)\times (0,T),\label{pro3.1}\\
&\mu \widetilde{P}_{tt}-\beta \widetilde{P}_{xx}+\gamma\beta\widetilde{V}_{xx}=0,  &(x,t)\in(0,L)\times (0,T),\label{pro3.11}
\end{align}
with the boundary conditions
\begin{align}
&\widetilde{V}(0,t)=\widetilde{P}(0,t)=0,&t\in(0,T),\\
&\alpha \widetilde{V}_{x}(L,t)-\gamma\beta \widetilde{P}_{x}(L,t)=0,& t\in (0,T),\\
&\beta \widetilde{P}_{x}(L,t)-\gamma\beta \widetilde{V}_{x}(L,t)=0,& t\in (0,T),
\end{align}
with the same initial conditions as in \eqref{1.3}, i.e.,
\begin{align}\label{pro3.2}
\left(\widetilde{V}(x,0),\widetilde{V}_{t}(x,0),\widetilde{P}(x,0),\widetilde{P}_{t}(x,0)\right)=\left(V_{0}(x),V_{1}(x),P_{0}(x),P_{1}(x)\right),x\in(0,L).
\end{align}
Problem \eqref{pro3.1}-\eqref{pro3.2} is defined in space $\mathcal{H}_{0}$, which is a closed subspace of $\mathcal{H}$,
\begin{align*}
\mathcal{H}_{0}:=H^{1}_{*}(0,L)\times {L}^{2}(0,L)\times\{0\}\times {{H}}^{1}_{*}(0,L)\times {{L}}^{2}(0,L)\times\{0\}.
\end{align*}
The solution of system $\{\mathcal{S}_{0}(t)\}$ with $U_{0}=\left(V_{0}, V_{1}, 0, P_{0}, P_{1}, 0\right)\in\mathcal{H}_{0}$ is given by
\begin{align*}
{\mathcal{S}_{0}(t)U_{0}}=\left(\widetilde{V}, \widetilde{V}_{t}, 0, \widetilde{P}, \widetilde{P}_{t}, 0\right).
\end{align*}
Note that $\{\mathcal{S}_{0}(t)\}$ defines a unitary group in $\mathcal{H}_{0}$. That is,
\begin{align*}
\|{\mathcal{S}_{0}(t)U_{0}}\|^{2}=\|U_{0}\|^{2}.
\end{align*}
\begin{lemma}\label{lem3.5}
Let $\widetilde{E}(t)$ be the energy of system \eqref{pro3.1}-\eqref{pro3.2}, then
\begin{align*}
\int^{T}_{0}\left(\left|\widetilde{{V}}_{t}(L,t)\right|^{2}+\left|\widetilde{{P}}_{t}(L,t)\right|^{2}\right)dt\leq C_{T}\widetilde{E}(0),
\end{align*}
where $C$ is a constant.
\end{lemma}
\begin{proof}
We multiply \eqref{pro3.1} and \eqref{pro3.11} by $x\widetilde{V}_{x}$ and $x\widetilde{P}_{x}$, respectively, and use a similar method in Lemma \ref{lem3.2} to get
\begin{align*}
\int^{T}_{0}\left(\left|\widetilde{{V}}_{t}(L,t)\right|^{2}+\left|\widetilde{{P}}_{t}(L,t)\right|^{2}\right)dt\leq\int^{T}_{0}\widetilde{E}(t)dt+C\widetilde{E}(T)+C\widetilde{E}(0),
\end{align*}
where $C$ is a constant, $\widetilde{E}(t)$ is the energy of system \eqref{pro3.1}-\eqref{pro3.2}. And combining with the fact that the energy of conservative system is always equal to the initial energy, we obtain
\begin{align}\label{3.82}
\int^{T}_{0}\left(\left|\widetilde{{V}}_{t}(L,t)\right|^{2}+\left|\widetilde{{P}}_{t}(L,t)\right|^{2}\right)dt\leq C_{T}\widetilde{E}(0),
\end{align}
where $C_{T}$ is a constant.
\end{proof}

Next, we give a lemma, which can help us to prove that the difference ${\mathcal{S}(t)}-{\mathcal{S}_{0}(t)}$ is a compact operator.
\begin{lemma}\label{lem3.4}
\textup{(\cite{mary2019})}
Let $R\in \mathbb{R}$, and the set $\mathcal{K}$ is defined by
\begin{align*}
\mathcal{K}=\left\{\partial^{a,\eta}_{t}f(t):f'\in L^{2}(0,T),\int^{T}_{0}|f'(t)|^{2}dt\leq R^{2}\right\}.
\end{align*}
Then, $\mathcal{K}$ is a compact subset in $L^{2}(0,T)$.
\end{lemma}

\begin{theorem}
The $C_{0}$-semigroup $\{\mathcal{S}(t)\}_{t\geq 0}$ is not exponentially stable.
\end{theorem}
\begin{proof}
To prove the lack of exponential stability, we shall use Theorem \ref{thm3.1}. Therefore, we need to show that the difference ${\mathcal{S}(t)}-{\mathcal{S}_{0}(t)}$ is a compact operator.

For any bounded sequence $U^{n}_{0}=\left(V^{n}_{0}, V^{n}_{1}, 0, P^{n}_{0}, P^{n}_{1}, 0\right)\in\mathcal{H}_0$, the sequence
\begin{align*}
U^{n}=\mathcal{S}(t){U^{n}_{0}}=\left(V^{n}, V^{n}_{t}, \phi^{n}_{1}, P^{n}, P^{n}_{t}, \phi^{n}_{2}\right)\in\mathcal{H}
\end{align*}
are bounded solutions of system \eqref{2.3}-\eqref{2.7}, and the sequence
\begin{align*}
\widetilde{U}^{n}=\mathcal{S}_{0}(t){U^{n}_{0}}=\left(\widetilde{V}^{n}, \widetilde{V}^{n}_{t}, 0, \widetilde{P}^{n},\widetilde{P}^{n}_{t}, 0\right)\in\mathcal{H}_0
\end{align*}
are bounded solutions of system \eqref{pro3.1}-\eqref{pro3.2}. Then, we can obtain
\begin{align*}
\left(\mathcal{S}(t)-\mathcal{S}_{0}(t)\right)U^{n}_{0}=\left(\widehat{V}^{n}, \widehat{V}^{n}_{t},  \phi^{n}_{1}, \widehat{P}^{n}, \widehat{P}^{n}_{t},  \phi^{n}_{2}\right)\in\mathcal{H}
\end{align*}
where
\begin{align*}
\widehat{V}^{n}_{x}=V^{n}_{x}-\widetilde{V}^{n}_{x},\quad \widehat{P}^{n}_{x}=P^{n}_{x}-\widetilde{P}^{n}_{x},&\quad\widehat{V}^{n}_{t}=V^{n}_{t}-\widetilde{V}^{n}_{t},\quad\widehat{P}^{n}_{t}=P^{n}_{t}-\widetilde{P}^{n}_{t},
\end{align*}
satisfies
\begin{align}
&\rho \widehat{V}^{n}_{tt}-\alpha \widehat{V}^{n}_{xx}+\gamma\beta \widehat{P}^{n}_{xx}=0,&&(x,t)\in  (0,L)\times (0,T), \label{3.5}\\
&\partial_{t}\phi^{n}_{1}(\xi,t)+\left(\xi^{2}+\eta\right)\phi^{n}_{1}(\xi,t)-{V}^{n}_{t}(L,t)\mu(\xi)=0,&&(\xi,t)\in(-\infty,+\infty)\times(0,+\infty),\label{3.51}\\
&\mu \widehat{P}^{n}_{tt}-\beta \widehat{P}^{n}_{xx}+\gamma\beta \widehat{V}^{n}_{xx}=0,  && (x,t)\in(0,L)\times (0,T),\label{3.52}\\
&\partial_{t}\phi^{n}_{2}(\xi,t)+\left(\xi^{2}+\eta\right)\phi^{n}_{2}(\xi,t)-{P}^{n}_{t}(L,t)\mu(\xi)=0,&&(\xi,t)\in(-\infty,+\infty)\times(0,+\infty),\label{3.53}
\end{align}
with the boundary conditions
\begin{align}
&\widehat{V}^{n}(0,t)=\widehat{P}^{n}(0,t)=0,& t\in (0,T),\\
&\alpha \widehat{V}^{n}_{x}(L,t)-\gamma\beta \widehat{P}^{n}_{x}(L,t)=-l_{1}\partial^{a,\eta}_{t}V^{n}(L,t),& t\in (0,T),\label{3.54}\\
&\beta \widehat{P}^{n}_{x}(L,t)-\gamma\beta \widehat{V}^{n}_{x}(L,t)=-l_{2}\partial^{a,\eta}_{t}P^{n}(L,t),& t\in (0,T),\label{3.55}
\end{align}
and the initial conditions
\begin{align}\label{3.56}
\left(\widehat{V}^{n}(x,0), \widehat{V}^{n}_{t}(x,0), \phi_1(0), \widehat{P}^{n}(x,0), \widehat{P}^{n}_{t}(x,0), \phi_2(0)\right)=\left(0,0,0,0,0,0\right),\quad x \in (0,L).
\end{align}
The energy associated with system \eqref{3.5}-\eqref{3.56} is defined by
\begin{align*}
\widehat{E}(t)=&\frac{1}{2}\int^{L}_{0}\left[\rho\left|\widehat{V}^{n}_{t}\right|^{2}+\alpha_{1}\left|\widehat{V}^{n}_{x}\right|^{2}+\mu\left|\widehat{P}^{n}_{t}\right|^{2}+\beta\left|\gamma \widehat{V}^{n}_{x}-\widehat{P}^{n}_{x}\right|^{2}\right] dx\\
&+\frac{\sin(a\pi)}{2\pi}\int^{\infty}_{-\infty}\left(l_{1}|\phi^{n}_{1}|^{2}+l_{2}|\phi^{n}_{2}|^{2}\right)d\xi.
\end{align*}
Multiplying \eqref{3.5}, \eqref{3.52}, \eqref{3.51} and \eqref{3.53} by $ \overline{\widehat{V}^{n}_{t}}$ , $\overline{\widehat{P}^{n}_{t}}$, $l_{1}\frac{\sin(a\pi)}{\pi}\phi^{n}_{1}$ and $l_{2}\frac{\sin(a\pi)}{\pi}\phi^{n}_{2}$, respectively, integrating by parts over $(0,L)$ and use the boundary conditions \eqref{3.54}-\eqref{3.55}, we can obtain
\begin{align*}
\frac{d}{dt}\widehat{E}(t)=-l_{1}\partial^{a,\eta}_{t}V^{n}(L,t)\widehat{V}^{n}_{t}(L,t)-l_{2}\partial^{a,\eta}_{t}P^{n}(L,t)\widehat{P}^{n}_{t}(L,t).
\end{align*}
By using the fact that $\widehat{V}^{n}_{t}=V^{n}_{t}-\widetilde{V}^{n}_{t}, \widehat{P}^{n}_{t}=P^{n}_{t}-\widetilde{P}^{n}_{t}$ and the energy associated with system \eqref{3.5}-\eqref{3.56}, we have
\begin{equation}\label{3.60}
\begin{aligned}
\frac{d}{dt}\widehat{E}(t)=&\frac{d}{dt}E(t)+l_{1}\partial^{a,\eta}_{t}V^{n}(L,t)\widetilde{{V}}^{n}_{t}(L,t)+l_{2}\partial^{a,\eta}_{t}P^{n}(L,t)\widetilde{{P}}^{n}_{t}(L,t)\\
=&-\frac{\sin(a\pi)}{\pi}\int^{\infty}_{-\infty}(\xi^{2}+\eta)\left(l_{1}|\phi^n_{1}|^{2}+l_{2}|\phi^n_{2}|^{2}\right)d\xi\\
&+l_{1}\partial^{a,\eta}_{t}V^{n}(L,t)\widetilde{{V}}^{n}_{t}(L,t)+l_{2}\partial^{a,\eta}_{t}P^{n}(L,t)\widetilde{{P}}^{n}_{t}(L,t).
\end{aligned}
\end{equation}
Integrating \eqref{3.60} over $[0,t]$ and using the initial conditions \eqref{3.56}, we obtain
\begin{equation}\label{3.70}
\begin{aligned}
\widehat{E}(t)&+\int^t_0\frac{\sin(a\pi)}{\pi}\int^{\infty}_{-\infty}(\xi^{2}+\eta)\left(l_{1}|\phi_{1}|^{2}+l_{2}|\phi_{2}|^{2}\right)d\xi d\tau\\
&=\int^{t}_{0} \left[l_{1}\partial^{a,\eta}_{t}V^{n}(L,\tau)\widetilde{{V}}^{n}_{t}(L,\tau)+l_{2}\partial^{a,\eta}_{t}P^{n}(L,\tau)\widetilde{{P}}^{n}_{t}(L,\tau)\right]d\tau.
\end{aligned}
\end{equation}

Suppose that initial data $\left({V}^{n}_{0}, {V}^{n}_{1}, 0, {P}^{n}_{0}, {P}^{n}_{1}, 0\right)\in\mathcal{H}_{0}$ is a bounded sequence. The solution of system \eqref{3.5}-\eqref{3.56} is denoted by $\left(\widehat{V}^{n},\widehat{V}^{n}_{t},\phi^{n}_{1},\widehat{P}^{n},\widehat{P}^{n}_{t},\phi^{n}_{2}\right)$. Since the dissipation of $\mathcal{S}(t)$ and conservation of $\mathcal{S}_{0}(t)$, the energy $\widehat{E}(t)$ associated with system \eqref{3.5}-\eqref{3.56} is non-increasing and bounded. Then, the above sequence have a subsequence converges weakly, and we still denote as
\begin{align}\label{3.80}
\left(\widehat{V}^{n},\widehat{V}^{n}_{t},\phi^{n}_{1},\widehat{P}^{n},\widehat{P}^{n}_{t},\phi^{n}_{2}\right).
\end{align}
Next, we will prove that the sequence \eqref{3.80} converges strongly. Lemmas \ref{lem3.2} and \ref{lem3.5} imply that $V^{n}_{t}(L,t)$, $P^{n}_{t}(L,t)$,$\widetilde{V}^{n}_{t}(L,t)$, and $\widetilde{P}^{n}_{t}(L,t)$ are bounded in $L^{2}(0,T)$. Using Lemma \ref{lem3.4}, we conclude that there exist strongly convergent subsequences $\partial^{a,\eta}_{t}V^{nk}(L,t)$, $\partial^{a,\eta}_{t}P^{nk}(L,t)$.
Therefore, the right hand side of the equality \eqref{3.70} converges strongly, and the solution
\begin{align*}
\left(\widehat{V}^{nk},\widehat{V}^{nk}_{t},\phi^{nk}_{1},\widehat{P}^{nk},\widehat{P}^{nk}_{t},\phi^{nk}_{2}\right)
\end{align*}
converges in norm. This implies that $\mathcal{S}(t)-\mathcal{S}_{0}(t)$ can map any bounded sequence to a strongly convergent subsequence, which means that $\mathcal{S}(t)-\mathcal{S}_{0}(t)$ is a compact operator from $\mathcal{H}_{0}$ to $\mathcal{H}$. Thanks to Theorem \ref{thm3.1}, we have proved the theorem.
\end{proof}
\section{Polynomial stability for piezoelectric beams without thermal effects}
In the case where $\{\mathcal{S}(t)\}_{t\geq0}$ is not exponentially stable, we look for a polynomial decay rate. In this section, we use the result of Borichev and Tomilov to show the polynomial stability of the piezoelectric system \eqref{2.3}-\eqref{2.8}.

\begin{theorem}\label{thm4.1}
\textup{({\cite{bor2010}})}
Assume that $\mathcal{A}$  is the infinitesimal generator of bounded $C_{0}-$ semigroup of contractions $\{\mathcal{S}(t)\}_{t\geq0}$  on Hilbert space $H$. If $i\mathbb{R}\subset \rho (\mathcal{A})$, then for any $k>0$, the following conditions are equivalent: \\
$(1)$ $\left\|\left(i\lambda I -\mathcal{A}\right)^{-1}\right\|_{\mathcal{L}(\mathcal{H})}=o\left(|\lambda |^{k}\right),\lambda\rightarrow \infty$;\\
$(2)$ $\left\|\mathcal{S}(t)\mathcal{A}^{-1}\right\|_{\mathcal{L}(\mathcal{H})}=o\left(t^{-\frac{1}{k}}\right),t\rightarrow \infty$.
\end{theorem}

First, let us consider the resolvent equation
\begin{align}\label{4.1}
i\lambda U-\mathcal{A}U=F,\quad\lambda\in\mathbb{R}.
\end{align}
Taking the real part of inner product between \eqref{4.1} and $U$ in $\mathcal{H}$, then we can obtain
\begin{align}\label{4.1+}
Re\langle -\mathcal{A}U, U \rangle_\mathcal{H}=Re\langle F, U \rangle_\mathcal{H}\leq C\|U\|_{\mathcal{H}}\|F\|_{\mathcal{H}},
\end{align}
where $F=(F_{1},F_{2},F_{3},F_{4},F_{5},F_{6})$. Rewrite the previous equation according to its components:
\begin{align}
i\lambda V-f&=F_{1},\label{4.11}\\
 i\lambda\rho f-{\alpha}V_{xx}+{\gamma\beta}P_{xx}&={\rho}F_{2},\label{4.12}\\
i\lambda \phi_{1}+\left(\xi^{2}+\eta\right)\phi_{1}(\xi,t)-f(L)\mu(\xi)&=F_{3},\label{4.13}\\
i\lambda P-g&=F_{4},\label{4.14}\\
i\lambda\mu g-{\beta}P_{xx}+{\gamma\beta}V_{xx}&={\mu}F_{5},\label{4.15}\\
i\lambda \phi_{2}+\left(\xi^{2}+\eta\right)\phi_{2}(\xi,t)-g(L)\mu(\xi)&=F_{6}.\label{4.16}
\end{align}

For further proof, we introduce some functionals as follows.
\begin{align*}
I_{V}&=\rho q(L)|f(L)|^{2}+\alpha_{1}q(L)|V_{x}(L,t)|^{2},\\
I_{P}&=\mu q(L)|g(L)|^{2}+\beta q(L)|(\gamma V_{x}-P_{x})(L,t)|^{2},\\
\mathcal{N}^{2}&=\int^{L}_{0}\rho |f|^{2} dx +\int^{L}_{0}\mu|g|^{2} dx +\int^{L}_{0}\alpha_{1}|V_{x}|^{2}dx +\int^{L}_{0}\beta|\gamma V_{x}-P_{x}|^{2}dx.
\end{align*}
\begin{lemma}\label{lem3.3}
Let $F=(F_{1},F_{2},F_{3},F_{4},F_{5},F_{6})\in\mathcal{H}$, $\lambda\in \mathbb{R}$, and ${U}=(V,f,,\phi_{1},P,g,\phi_{2})\in \mathcal{D}(\mathcal{A})$ satisfies  $i\lambda{U}-\mathcal{A}{U}={F}$. Then for $q\in C^{2}([0,L])$, $q(0)=0$, we have
\begin{align*}
&I_{V}+I_{P}-\int^{L}_{0}\rho q_{x}|f|^{2} dx -\int^{L}_{0}\mu q_{x}|g|^{2} dx -\int^{L}_{0}\alpha_{1} q_{x}|V_{x}|^{2}dx -\int^{L}_{0}\beta q_{x}|\gamma V_{x}-P_{x}|^{2}dx\\
=&-R_{1}-R_{2},
\end{align*}
where
\begin{align*}
R_{1}=Re\int^{L}_{0}\left(2\mu q F_{4}\overline{P}_{x}+2\mu q \overline{F}_{3,x}g \right)dx,\\
R_{2}=Re\int^{L}_{0}\left(2\rho q F_{2}\overline{V}_{x}+2\rho q \overline{F}_{1,x}f\right) dx.
\end{align*}
\end{lemma}
\begin{proof}
Let us multiply the equation \eqref{4.12} by $q \overline{V}_{x}$ and integrating on $[0,L]$, we have
\begin{align}\label{4.1'}
\int^{L}_{0}\left(- i\lambda \rho f q\overline{V}_{x}+\alpha qV_{xx}\overline{V}_{x}-\gamma\beta qP_{xx}\overline{V}_{x} \right)dx=-\int^{L}_{0}\rho qF_{2}\overline{V}_{x} dx.
\end{align}
Note that when we use \eqref{4.11}, the first term in \eqref{4.1'} can be rewritten as
\begin{align}\label{4.2}
\int^{L}_{0} -i\lambda \rho f q\overline{V}_{x}dx=\int^{L}_{0} \overline{(i\lambda V_{x})}\rho q f dx=\int^{L}_{0}\rho qf\overline{(f_{x}+F_{1,x})}dx.
\end{align}
Multiplying \eqref{4.15} by $q \overline{P}_{x}$ and integrating on $[0,L]$, we arrive at
\begin{align}\label{4.3}
\int^{L}_{0}\left(- i\lambda\mu g q\overline{P}_{x}+\beta qP_{xx}\overline{P}_{x}-\gamma\beta qV_{xx}\overline{P}_{x}\right) dx=-\int^{L}_{0}\mu qF_{5}\overline{P}_{x} dx.
\end{align}
Here we use equality \eqref{4.14}, then the first term in \eqref{4.3} can be rewritten as
\begin{align}\label{4.4}
\int^{L}_{0}- i\lambda \mu g q\overline{P}_{x}dx=\int^{L}_{0} \overline{(i\lambda P_{x})}\mu q g dx=\int^{L}_{0}\mu qg\overline{(g_{x}+F_{3,x})}dx.
\end{align}
Substituting \eqref{4.2}, \eqref{4.4} in \eqref{4.1'}, \eqref{4.3} respectively, then adding \eqref{4.1'} and \eqref{4.3}, we get
\begin{align}\label{4.5}
&\int^{L}_{0} \rho q \frac{d}{dx} |f|^{2} dx +\int^{L}_{0} \alpha_{1} q \frac{d}{dx}|V_{x}|^{2} dx +\int^{L}_{0} \mu q \frac{d}{dx}|g|^{2}dx +\int^{L}_{0}\beta q \frac{d}{dx} |\gamma V_{x}-P_{x}|^{2} dx
\nonumber\\
=&Re\int^{L}_{0}\left(-2\mu q F_{4}\overline{P}_{x}-2\mu q \overline{F}_{3,x}g -2\rho q F_{2}\overline{V}_{x}-2\rho q f\overline{F}_{1,x}\right)dx.
\end{align}
Finally, integrating by part, we get the equality in the lemma.
\end{proof}
\begin{lemma}\label{lem4.4}
Let $\mathcal{N}$, $I_{V}, I_{P}$ be functionals defined above, then we have
\begin{align}\label{4.40}
\mathcal{N}^{2}\leq C\left(I_{V}+I_{P}+\|F\|^{2}_{\mathcal{H}}\right),
\end{align}
where $C$ is a constant.
\end{lemma}
\begin{proof}
Let $q(x)=x, x\in [0,L]$. From the result of Lemma \ref{lem3.3}, we have
\begin{align}\label{4.41}
\mathcal{N}^{2}= I_{V}+I_{P}+R_{1}+R_{2}.
\end{align}
From the definition of $R_{1},R_{2}$, we conclude that
\begin{align}\label{4.42}
|R_{1}|\leq C\mathcal{N}\|F\|_{\mathcal{H}}, |R_{2}|\leq C\mathcal{N}\|F\|_{\mathcal{H}}.
\end{align}
By using the Cauchy-Schwartz inequality and the inequality \eqref{4.42}, we verify that the inequality \eqref{4.40} is valid.
\end{proof}
\begin{theorem}\label{iR}
Let $\rho(\mathcal{A})$ is the resolvent set of the operator $\mathcal{A}$, then $i\mathbb{R}\in\rho(\mathcal{A})$.
\end{theorem}
\begin{proof}
In Section 2, we have proved that $0\in \rho(\mathcal{A})$, then we know that the set
\begin{align*}
\mathcal{M}=\left\{\beta>0:(-i\beta,i\beta)\subset \rho (\mathcal{A})\right\}\neq \emptyset.
\end{align*}
If $\sup\limits_{\beta>0}{\mathcal{M}}=+\infty$, the conclusion is obviously correct. Next, we will consider $\sup \limits_{\beta>0}{\mathcal{M}}< +\infty$ by using reduction to absurdity. Assume that there exists $\lambda>0$ such that $\sup \limits_{\beta>0}{\mathcal{M}}=\lambda <+\infty$. We can clearly find that $\lambda\notin \mathcal{M}$. Therefore, there exist $\lambda_{n} \in \mathcal{M}$ and $\overline{F}_{n}\in \mathcal{H}$ with $\|\overline{F}_{n}\|=1$ such that
\begin{align*}
\left\|(i\lambda_{n} I-\mathcal{A})^{-1}\overline{F}_{n}\right\|_{\mathcal{H}}\rightarrow \infty.
\end{align*}
Let $ \overline{U}_{n}=(i\lambda_{n}I-\mathcal{A})^{-1}\overline{F}_{n}$, which implies that $i\lambda_{n}\overline{U}_{n}-\mathcal{A}\overline{U}_{n}=\overline{F}_{n}$. Define $U_{n}=\frac{\overline{U}_{n}}{\|(i\lambda _{n}-\mathcal{A})^{-1}\overline{F}_{n}\|_{\mathcal{H}}}$, then $U_{n}$ satisfies
\begin{align*}
i\lambda_{n}U_{n}-\mathcal{A}U_{n}=F_{n},
\end{align*}
where $F_{n}=\frac{\overline{F}_{n}}{\|(i\lambda _{n}-\mathcal{A})^{-1}\overline{F}_{n}\|_{\mathcal{H}}}$. Since $\|\overline{F}_{n}\|_{\mathcal{H}}=1$ and $\|(i\lambda_n I -\mathcal{A})^{-1}\overline{F}_{n}\|_{\mathcal{H}}\rightarrow \infty$, we have $F_{n}\rightarrow 0$. Finding inner product with $U_{n}$ on $\mathcal{H}$, we have
\begin{align*}
i\lambda_{n}\|U_{n}\|^{2}_{\mathcal{H}}-\langle \mathcal{A}U_{n},U_{n}\rangle _{\mathcal{H}}=\langle F_{n},U_{n}\rangle _{\mathcal{H}}.
\end{align*}
By taking the real part and using the fact of $F_{n}\rightarrow 0$, we obtain
\begin{align*}
-Re\langle \mathcal{A}U_{n},U_{n}\rangle _\mathcal{H}=Re\langle F_{n},U_{n}\rangle _{\mathcal{H}}\rightarrow 0,
\end{align*}
that is
\begin{align*}
\int^{\infty}_{-\infty}(\xi^{2}+\eta)\left(l_{1}|\phi_{1,n}|^{2}+l_{2}|\phi_{2,n}|^{2}\right)d\xi\rightarrow 0.
\end{align*}
 By using \eqref{2.301} and \eqref{2.401}, we have  $f_{n}(L),g_{n}(L)\rightarrow 0$. Using the fact that $\phi_{1,n}, \phi_{2,n}\rightarrow 0$ and  \eqref{2.501}, \eqref{2.502}, we then obtain $V_{x,n}(L,t),P_{x,n}(L,t)\rightarrow 0$, which concludes that $I_{V}+I_{P}\rightarrow 0$. Further, with the help of the fact that $F_{n}, \phi_{1,n}, \phi_{2,n}\rightarrow 0$ and lemma \ref{lem4.4}, we have $U_{n}\rightarrow 0$.

There is a contradiction between $U_{n}\rightarrow 0$ and $\|U_{n}\|_{\mathcal{H}}=1$. Therefore, by using reduction to absurdity, the proof is completed.
\end{proof}

Now, we will use the result of Borichev and Tomilov theorem to show the polynomial stability for piezoelectric beams without thermal effects.

\begin{theorem}\label{thm4.5}
The piezoelectric system \eqref{2.3}-\eqref{2.7} with tip body decays polynomially as
\begin{align*}
\|{U}(t)\|_{\mathcal{H}}\leq \|{U}_{0}\|_{\mathcal{D}(\mathcal{A})}\frac{C}{{t}^{\frac{1}{1-a}}}.
\end{align*}
\end{theorem}
\begin{proof}
To obtain the above inequality, we first need to consider the relationship between $|f(L)|^{2},|g(L)|^{2}$ and $\|F\|$. Multiplying \eqref{4.16} by $\frac{|\xi|}{\left({\lambda+\xi^{2}+\eta}\right)^2}$, we can obtain
\begin{align*}
g(L)\frac{\mu(\xi)|\xi|}{\left(\lambda+\xi^{2}+\eta\right)^2}=-\frac{F_{6}(\xi)|\xi|}{\left(\lambda+\xi^{2}+\eta\right)^2}+\frac{i\lambda|\xi|}{\left(\lambda+\xi^{2}+\eta\right)^2}\phi_{2}
+\frac{(\xi^{2}+\eta)|\xi|}{\left(\lambda+\xi^{2}+\eta\right)^2}\phi_{2}.
\end{align*}
Integrating over $(-\infty,+\infty)$ with respect to the variable $\xi$, and using the fact that $\mu(\xi)=|\xi|^{(2a-1)/2}$, we have
\begin{align*}
|g(L)|\int^{+\infty}_{-\infty}\frac{|\xi|^{a+\frac{1}{2}}}{(\lambda+\xi^2+\eta)^2}d\xi\leq \int^{+\infty}_{-\infty}\frac{|\phi_2(\xi)||\xi|}{|\lambda+\xi^2+\eta|}d\xi+\int^{+\infty}_{-\infty}\frac{|F_6(\xi)||\xi|}{(\lambda+\xi^2+\eta)^2}d\xi.
\end{align*}
Here, we use the Cauchy-Shariwz inequality, then the above inequation can be rewritten as
\begin{align*}
|g(L)|\int^{+\infty}_{-\infty}\frac{|\xi|^{a+\frac{1}{2}}}{(\lambda+\xi^2+\eta)^2}d\xi\leq & \left(\int^{+\infty}_{-\infty}\frac{1}{(\lambda+\xi^2+\eta)^2}d\xi\right)^{\frac{1}{2}} \left(\int^{+\infty}_{-\infty}{|\phi_2(\xi)|^2\xi^2}d\xi\right)^\frac{1}{2}\\
&+\left(\int^{+\infty}_{-\infty}\frac{\xi^2}{(\lambda+\xi^2+\eta)^4}d\xi\right)^{\frac{1}{2}}
\left(\int^{+\infty}_{-\infty}|F_6(\xi)|^2d\xi\right)^\frac{1}{2}.
\end{align*}
Then, using Lemma \ref{nuse3}, we get
\begin{align*}
|g(L)|\left(c_1(\lambda+\eta)^{\frac{a}{2}-\frac{5}{4}}\right)\leq \sqrt{\frac{\pi}{2}}\frac{1}{(\lambda+\eta)^{\frac{3}{4}}} \left(\int^{+\infty}_{-\infty}{|\phi_2(\xi)|^2\xi^2}d\xi\right)^\frac{1}{2}+ \frac{\sqrt{\pi}}{4}\frac{1}{(\lambda+\eta)^{\frac{5}{4}}}\left(\int^{+\infty}_{-\infty}|F_6(\xi)|^2d\xi\right)^\frac{1}{2},
\end{align*}
which implies that
\begin{align*}
|g(L)|&\leq
\sqrt{\frac{\pi}{2}}\frac{1}{c_1(\lambda+\eta)^{\frac{a-1}{2}}} \left(\int^{+\infty}_{-\infty}{|\phi_2(\xi)|^2\xi^2}d\xi\right)^\frac{1}{2}+ \frac{\sqrt{\pi}}{4}\frac{1}{c_1(\lambda+\eta)^{\frac{a}{2}}}\left(\int^{+\infty}_{-\infty}|F_6(\xi)|^2d\xi\right)^\frac{1}{2}\\
&\leq C|\lambda|^{\frac{1-a}{2}}\|U\|^{\frac{1}{2}}_{\mathcal{H}}\|F\|^{\frac{1}{2}}_{\mathcal{H}}+C|\lambda|^{-\frac{a}{2}}\|F\|_{\mathcal{H}}.
\end{align*}
Now, we can get that
\begin{align*}
|g(L)|^{2}\leq C|\lambda|^{-a}\|F\|^{2}_{\mathcal{H}}+C|\lambda|^{1-a}\|U\|_{\mathcal{H}}\|F\|_{\mathcal{H}}.
\end{align*}
Similarly, we can get $|f(L)|^{2}\leq C|\lambda|^{-a}\|F\|^{2}_{\mathcal{H}}+C|\lambda|^{1-a}\|U\|_{\mathcal{H}}\|F\|_{\mathcal{H}}$.
On the other hand, we need to consider the relationship between $\alpha_{1}|V_{x}(L)|^{2}, \beta|(\gamma V_{x}-P_{x})(L)|^{2}$ and $\|F\|_{\mathcal{H}}$.
Thanks to \eqref{2.501} and \eqref{2.502}, it is straightforward to verify that
\begin{align*}
\alpha_{1}V_{x}(L)&=-l_{1}\frac{\sin(a\pi)}{\pi}\int^{\infty}_{-\infty}\mu(\xi)\phi_{1}(\xi,t)d\xi-\gamma l_{2}\frac{\sin(a\pi)}{\pi}\int^{\infty}_{-\infty}\mu(\xi)\phi_{2}(\xi,t)d\xi,\\
\beta(\gamma V_{x}-P_{x})(L)&=-l_{2}\frac{\sin(a\pi)}{\pi}\int^{\infty}_{-\infty}\mu(\xi)\phi_{2}(\xi,t)d\xi.
\end{align*}
Then, by employing Lemma \ref{nuse2} and \eqref{4.1+}, we obtain
\begin{align*}
\alpha_{1}|V_{x}(L)|^{2}\leq C\|U\|_{\mathcal{H}}\|F\|_{\mathcal{H}},\\
\beta|(\gamma V_{x}-P_{x})(L)|^{2}\leq C\|U\|_{\mathcal{H}}\|F\|_{\mathcal{H}}.
\end{align*}
Now, by recalling the definition of norm in $\mathcal{H}$ together with the result of Lemma \ref{lem4.4}, we arrive at
\begin{align*}
\|{U}\|^{2}_{\mathcal{H}}\leq C|\lambda|^{-a}\|F\|^{2}_{\mathcal{H}}+C|\lambda|^{1-a}\|U\|_{\mathcal{H}}\|F\|_{\mathcal{H}}
\end{align*}
for $|\lambda|>1$, which implies
\begin{align*}
\|{U}\|^{2}_{\mathcal{H}}\leq C|\lambda|^{2-2a}\|F\|^{2}_{\mathcal{H}}.
\end{align*}
Finally, we can get the result of polynomial stability  by using Theorem \ref{thm4.1}.
\end{proof}

\section{Well-posedness for piezoelectric beams with thermal effects}
In the previous sections, we get the fact that the system is not exponentially stable when the magnetic effect is in fractional form and the temperature is not being considered. In this section, we can obtain that the system is exponentially stable if we consider the (dissipative) effect of heat conduction under the Fourier law in piezoelectric beam system. Using Theorem \ref{thm2.1}, system \eqref{problem1.2}-\eqref{1.6} can be rewritten as the next augmented model
\begin{align}
&\rho V_{tt}-\alpha V_{xx}+\gamma\beta P_{xx}+\delta\theta_{x}=0,&&(x,t)\in  (0,L)\times (0,T), \label{5.11}\\
&\partial_{t}\phi_{1}(\xi,t)+\left(\xi^{2}+\eta\right)\phi_{1}(\xi,t)-V_{t}(L,t)\mu(\xi)=0,&& (\xi,t)\in (-\infty,+\infty)\times (0,+\infty),\label{5.12}\\
&\mu P_{tt}-\beta P_{xx}+\gamma\beta V_{xx}=0,  && (x,t)\in(0,L)\times (0,T),\label{5.13}\\
&\partial_{t}\phi_{2}(\xi,t)+\left(\xi^{2}+\eta\right)\phi_{2}(\xi,t)-P_{t}(L,t)\mu(\xi)=0,&&(\xi,t)\in (-\infty,+\infty)\times (0,+\infty),\label{5.14}\\
&c\theta_{x}-\kappa\theta_{xx}+\delta V_{xt}=0, &&(x,t)\in(0,L)\times (0,T),
\end{align}
with the boundary conditions
\begin{align}
&V(0,t)=P(0,t)=\theta_x(0,t)=\theta(L,t)=0,& t\in (0,T),\label{bc1}\\
&\alpha V_{x}(L,t)-\gamma\beta P_{x}(L,t)=-l_{1}\frac{\sin(a\pi)}{\pi}\int^{\infty}_{-\infty}\mu(\xi)\phi_{1}(\xi,t)d\xi,& t\in (0,T),\label{5.15}\\
&\beta P_{x}(L,t)-\gamma\beta V_{x}(L,t)=-l_{2}\frac{\sin(a\pi)}{\pi}\int^{\infty}_{-\infty}\mu(\xi)\phi_{2}(\xi,t)d\xi,& t\in (0,T),\label{5.16}
\end{align}
and the initial conditions
\begin{align}\label{5.17}
\left(V(x,0), V_{t}(x,0), \phi_{1}(0), P(x,0), P_{t}(x,0),\phi_{2}(0),\theta(x,0)\right)=\left(V_{0},V_{1},\phi_{01},P_{0},P_{1},\phi_{02},\theta_{0}\right)\quad x \in (0,L).
\end{align}

Let us start by defining the phase spaces depending on each boundary condition in \eqref{1.5}
\begin{align*}
\mathcal{H}_{1}:=H^{1}_{*}(0,L)\times {L}^{2}(0,L)\times{L}^{2}(-\infty,+\infty)\times {{H}}^{1}_{*}(0,L)\times {{L}}^{2}(0,L)\times{L}^{2}(-\infty,+\infty)\times H^1_c(0,L),
\end{align*}
such that $ {{H}}^{1}_{*}(0,L)=\left\{f\in{{H}}^{1}(0,L):f(0)=0\right\}$, $H^1_c(0,L)=\left\{f\in{{H}}^{1}(0,L):f_x(0)=f(L)=0\right\}$, equipped with the inner product
\begin{align*}
\langle{U}_{1},{U}_{2}\rangle_{\mathcal{H}_{1}}=&\int^{L}_{0}\left[\rho f_{1}\overline{f}_{2}+\mu g_{1}\overline{g}_{2}+\alpha_{1}V_{1,x}\overline{V}_{2,x}+\beta(\gamma V_{1,x}-P_{1,x})\overline{(\gamma V_{2,x}-P_{2,x})}+c\theta_{1}\overline{\theta}_{2}\right]dx \\
&+\frac{\sin(a\pi)}{\pi}\int^{+\infty}_{-\infty}\left(l_{1}\phi_{1,1}\overline{\phi}_{1,2}+l_{2}\phi_{2,1}\overline{\phi}_{2,2}\right)dx,
\end{align*}
where ${{U}}_{i}=\left(V_{i},f_{i}, \phi_{1,i}, P_{i}, g_{i},\phi_{2,i},\theta_{i}\right)\in \mathcal{{H}}_{1},i=1,2$.
Set the vector function
\begin{align*}
{{U}}=\left(V,V_{t}, \phi_{1}, P, P_{t}, \phi_{2},\theta\right)^{T}.
\end{align*}
Transforming to a first-order system again, we can rewrite system \eqref{5.11}-\eqref{5.17} as
\begin{equation} \label{5.20}
\left\{
\begin{aligned}
&{{U}}_{t}={\mathcal{A}_{1}}{{U}},\\
&{{U}}(0)={{U}}_{0},
\end{aligned}
\right.
\end{equation}
where ${{U}}_{0}=\left(V_{0},V_{1}, \phi_{1,0}, P_{0}, P_{1}, \phi_{2,0},\theta_{0}\right)^{T}$ and the operator $\mathcal{A}_{1}:\mathcal{D}(\mathcal{A}_{1})\subset \mathcal{H}_{1}\rightarrow \mathcal{H}_{1}$ is given by
\begin{align*}
\mathcal{A}_{1}\begin{bmatrix} V \\ f \\ \phi_{1} \\ P \\ g\\ \phi_{2} \\ \theta \end{bmatrix}=
\begin{bmatrix}
   f \\
  \frac{\alpha}{\rho}V_{xx} -\frac{\gamma\beta}{\rho}P_{xx}-\frac{\delta}{\rho} \theta_{x} \\
  -\left(\xi^{2}+\eta\right)\phi_{1}(\xi,t)+f(L,t)\mu(\xi)\\
  g \\
  -\frac{\gamma\beta}{\mu}V_{xx} + \frac{\beta}{\mu}P_{xx} \\
  -\left(\xi^{2}+\eta\right)\phi_{2}(\xi,t)+g(L,t)\mu(\xi)\\
  \frac{k}{c}\theta_{xx}-\frac{\delta}{c} f_{x}
\end{bmatrix}
,
\end{align*}
where the domain of the operator $\mathcal{A}_{1}$ is given by
\begin{align*}
\mathcal{D}(\mathcal{A}_{1}):=\bigg\{&{{U}}\in \mathcal{H};  V,P \in {H}^{2}(0,L)\cap H^{1}_{*}(0,L), f,g\in H^{1}_{*}(0,L), \theta\in H^2(0,L)\cap H^{1}_c(0,L), \\&|\xi|\phi_{1},|\xi|\phi_{2}\in L^{2}(-\infty,+\infty),-\left(\xi^{2}+\eta\right)\phi_{1}(\xi,t)+f(L,t)\mu(\xi)\in L^{2}(-\infty,+\infty),\\
&-\left(\xi^{2}+\eta\right)\phi_{2}(\xi,t)+g(L,t)\mu(\xi)\in L^{2}(-\infty,+\infty)\bigg\}.
\end{align*}
with ${{U}}=\left(V,f, \phi_{1}, P, g, \phi_{2},\theta\right)$.

The energy of system \eqref{5.11}-\eqref{5.17} is given by
\begin{align}
E_{1}(t)=&\frac{1}{2}\int^{L}_{0}\left[\rho\left|V_{t}\right|^{2}+\alpha_{1}\left|V_{x}\right|^{2}+\mu\left|P_{t}\right|^{2}+\beta\left|\gamma V_{x}-P_{x}\right|^{2}+c|\theta|^{2}\right] dx\\
&+\frac{\sin(a\pi)}{2\pi}\int^{\infty}_{-\infty}\left(l_{1}|\phi_{1}|^{2}+l_{2}|\phi_{2}|^{2}\right)d\xi.
\end{align}
Multiplying \eqref{5.11}, \eqref{5.13}, \eqref{5.15} by $ V_{t}$ , $P_{t}$ and $\theta$ respectively and integrating on $(0,L)$,  multiplying \eqref{5.12}, \eqref{5.14} by $l_{1}\frac{\sin(a\pi)}{\pi}\phi_{1}$ and $l_{2}\frac{\sin(a\pi)}{\pi}\phi_{2}$ respectively and integrating on $\mathbb{R}$, we have
\begin{align}\label{5.21}
\frac{d}{dt}E_{1}(t)=-\frac{\sin(a\pi)}{\pi}\int^{\infty}_{-\infty}(\xi^{2}+\eta)\left(l_{1}|\phi_{1}|^{2}+l_{2}|\phi_{2}|^{2}\right)d\xi-k\int^{L}_{0}|\theta_{x}|^{2} dx.
\end{align}

\begin{theorem}\label{them5.1}
Let ${U}_{0}\in \mathcal{D}(\mathcal{A}_{1})$, there exists a unique solution ${U}(t)=\mathcal{S}_{\mathcal{A}_{1}}(t){U}_{0}$ of \eqref{5.20} such that
\begin{align*}
{U}\in{C}\left([0,\infty);\mathcal{D}(\mathcal{A}_{1})\right)\cap{C}^{1} \left([0,\infty); \mathcal{H}_{1}\right).
\end{align*}
\end{theorem}
\begin{proof}
It relies on Lumer-Phillips theorem, which ensures that the operator $\mathcal{A}_{1}$ is the infinitesimal generator of a $C_{0}-$ semigroup $\{\mathcal{S}_{\mathcal{A}_{1}}(t)\}_{t\geq0}$. From \eqref{5.21}, we have shown that the operator $\mathcal{A}_{1}$ is dissipative, closed and dense.
So we need to prove that $0\in \rho(\mathcal{A})$. For any ${F}=\left(F_{1},F_{2},F_{3},F_{4},F_{5},F_{6},F_{7}\right)\in \mathcal{H}_{1}$, we will prove there exists a ${{U}}=\left(V,f, \phi_{1}, P, g, \phi_{2},\theta\right)$ such that
\begin{align*}
-\mathcal{A}{U}={F}.
\end{align*}
Equivalently, we shall consider the existence of unique solution of system
\begin{align}
-f&=F_{1}\label{5.31}\\
 -{\alpha}V_{xx}+{\gamma\beta}P_{xx}+\delta\theta_{x}&={\rho}F_{2}\label{5.32}\\
\left(\xi^{2}+\eta\right)\phi_{1}(\xi,t)-f(L)\mu(\xi)&=F_{3}\label{5.33}\\
-g&=F_{4}\label{5.34}\\
-{\beta}P_{xx}+{\gamma\beta}V_{xx}&={\mu}F_{5}\label{5.35}\\
\left(\xi^{2}+\eta\right)\phi_{2}(\xi,t)-g(L)\mu(\xi)&=F_{6}\label{5.36}\\
-k\theta_{xx}+\delta f_{x}&=cF_{7}.\label{5.37}
\end{align}
That is, by \eqref{5.31} and \eqref{5.34} it follows that $f,g\in H^{1}_{*}(0,L)$ and
\begin{align}\label{2.13}
f=-F_{1},\quad & \quad g=-F_{4}.
\end{align}
Then combining \eqref{2.13} with \eqref{5.33} and \eqref{5.36}, we have that
\begin{align}\label{2.130}
\phi_{1}=\frac{-F_{1}(L)\mu(\xi)+F_{3}}{\xi^{2}+\eta},\quad\quad\phi_{2}=\frac{-F_{4}(L)\mu(\xi)+F_{6}}{\xi^{2}+\eta}.
\end{align}
Thanks to \eqref{2.13}, \eqref{5.37} and the boundary condition \eqref{bc1}, we can get that
\begin{align*}
\theta=\frac{1}{k}\int^{L}_{x}\int^{y}_{0}\left(\delta F_{1,x}+cF_{7}\right)(\xi)d\xi dy.
\end{align*}
By employing Lemma\ref{nuse1}, we conclude that $\phi_{i}\in L^{2}(R)$. Then, we need to prove the existence of unique solution of the system
\begin{equation}\label{2.14}
\left\{
\begin{aligned}
& {\alpha}V_{xx}-{\gamma\beta}P_{xx}=-\rho F_{2}+\frac{\delta}{k}\int^{x}_{0}\left(-\delta F_{1,x}-cF_{7}\right)(\xi)d\xi,\\
&\beta P_{xx}-\gamma\beta V_{xx}=-\mu F_{5} ,\\
&\alpha V_{x}(L,t)-\gamma\beta P_{x}(L,t)=-l_{1}\frac{\sin(a\pi)}{\pi}\int^{\infty}_{-\infty}\mu(\xi)\frac{-F_{1}(L)\mu(\xi)+F_{3}(\xi)}{\xi^{2}+\eta}d\xi,\\
&\beta P_{x}(L,t)-\gamma\beta V_{x}(L,t)=-l_{2}\frac{\sin(a\pi)}{\pi}\int^{\infty}_{-\infty}\mu(\xi)\frac{-F_{4}(L)\mu(\xi)+F_{6}(\xi)}{\xi^{2}+\eta}d\xi.
\end{aligned}
\right.
\end{equation}
By using the Lax-Milgram theorem, we know that there exists a solution $(V,P,\theta)\in {H}^{1}_{*}(0,L)\times {H}^{1}_{*}(0,L)\times H^1_c(0,L)$. Combining \eqref{2.14}, \eqref{2.13} and \eqref{2.130}, we obtain that $0\in \rho(\mathcal{A})$. Thus, the proof of the theorem is completed.
\end{proof}

\section{Exponential stability for piezoelectric beams with thermal effects}
In this section, we will show that the system of piezoelectric beams with Fourier's law is exponentially stable by constructing four functionals.
Before we begin the proof, we give a useful inequality. This inequality is proved by Lemma \ref{nuse1} and Young's inequality, which states the relationship between the dissipation term $l_{i}\frac{\sin(a\pi)}{\pi}\int^{\infty}_{-\infty}\mu(\xi)\phi_{i}(\xi,t)d\xi$  at the boundary $x=L$ and the term $\frac{\sin(a\pi)}{\pi}\int^{\infty}_{-\infty}(\xi^{2}+\eta)\left(l_{1}|\phi_{1}|^{2}+l_{2}|\phi_{2}|^{2}\right)d\xi$ in the energy derivative \eqref{5.21}.

\begin{lemma}
Let $\mu(\xi), \phi_{i}(\xi,t), a, \eta $ be the functions mentioned above, and $l_i\geq 0$. Then
\begin{align}\label{nuse2}
\left[l_{i}\frac{\sin(a\pi)}{\pi}\int^{\infty}_{-\infty}\mu(\xi)\phi_{i}(\xi,t)d\xi\right]^{2}\leq
M\frac{\sin(a\pi)}{\pi}\int^{\infty}_{-\infty}(\xi^{2}+\eta)\left(l_{1}|\phi_{1}|^{2}+l_{2}|\phi_{2}|^{2}\right)d\xi.
\end{align}
where $M$ is a constant.
\end{lemma}
\begin{proof}
First, the left hand side of the above inequality can be written as
\begin{align*}
\left[l_{1}\frac{\sin(a\pi)}{\pi}\int^{\infty}_{-\infty}\mu(\xi)\phi_{1}(\xi,t)d\xi\right]^{2}=
\left[l_{1}\frac{\sin(a\pi)}{\pi}\right]^{2}\left(\int^{\infty}_{-\infty}\frac{\mu(\xi)}{\sqrt{\xi^{2}+\eta+|\lambda|}}\phi_{1}(\xi,t){\sqrt{\xi^{2}+\eta+|\lambda|}}d\xi\right)^{2}
\end{align*}
Thanks to H\"{o}lder's inequality and Lemma \ref{nuse1}, it is straightforward to verify that
\begin{align*}
&\left(\int^{\infty}_{-\infty}\frac{\mu(\xi)}{\sqrt{\xi^{2}+\eta+|\lambda|}}\phi_{1}(\xi,t){\sqrt{\xi^{2}+\eta+|\lambda|}}d\xi\right)^{2}\\ & \leq
\left(\int^{\infty}_{-\infty}\frac{\mu^{2}(\xi)}{{\xi^{2}+\eta+|\lambda|}}d\xi\right)\left(\int^{\infty}_{-\infty}\phi^{2}_{1}(\xi,t)\left({\xi^{2}+\eta+|\lambda|}\right)d\xi\right)\\
&=\frac{\pi}{\sin(a\pi)}(\eta+|\lambda|)^{a-1}\left(\int^{\infty}_{-\infty}\phi^{2}_{1}(\xi,t)\left({\xi^{2}+\eta+|\lambda|}\right)d\xi\right).
\end{align*}
That is,
\begin{align*}
\left[l_{1}\frac{\sin(a\pi)}{\pi}\int^{\infty}_{-\infty}\mu(\xi)\phi_{1}(\xi,t)d\xi\right]^{2}\leq \frac{\sin(a\pi)}{\pi}(\eta+|\lambda|)^{a-1}\int^{\infty}_{-\infty}l_{1}^{2}\phi^{2}_{1}(\xi,t)\left({\xi^{2}+\eta+|\lambda|}\right)d\xi.
\end{align*}
 Since $0$ belongs to the set $E$, $\lambda$ can be taken as $0$. And we definite $M_{1}=\eta^{a-1}l_1$, then we have
\begin{align*}
\left[l_{1}\frac{\sin(a\pi)}{\pi}\int^{\infty}_{-\infty}\mu(\xi)\phi_{1}(\xi,t)d\xi\right]^{2}\leq
M_{1}\frac{\sin(a\pi)}{\pi}\int^{\infty}_{-\infty}(\xi^{2}+\eta)\left(l_{1}|\phi_{1}|^{2}\right)d\xi.
\end{align*}
Using the same method and taking $M_{2}=\eta^{a-1}l_2$, we can obtain
\begin{align*}
\left[l_{2}\frac{\sin(a\pi)}{\pi}\int^{\infty}_{-\infty}\mu(\xi)\phi_{2}(\xi,t)d\xi\right]^{2}\leq
M_{2}\frac{\sin(a\pi)}{\pi}\int^{\infty}_{-\infty}(\xi^{2}+\eta)\left(l_{2}|\phi_{2}|^{2}\right)d\xi.
\end{align*}
Finally, by choosing $M=\max\{M_{1},M_{2}\}$, we completely prove inequality \eqref{nuse2}.
\end{proof}
\begin{lemma}\label{lem6.1}
Let
\begin{align*}
I_{1}=\int^{L}_{0}\left( \rho V_{t}V+\mu P_{t}P \right)dx,
\end{align*}
then
\begin{align*}
\frac{d}{dt}I_{1}\leq& -\left(\alpha_{1}-(1+\gamma^{2})\eta_{1}\right)\int^{L}_{0}|V_{x}|^{2}dx-\left(\beta-\eta_{1}\right)\int^{L}_{0}|\gamma V_{x}-P_{x}|^{2}dx
+\frac{\delta^2 C_p}{2\eta_{1}}\int^{L}_{0}|\theta_{x}|^{2}dx\\
&+\frac{ML}{\eta_{1}}\frac{\sin(a\pi)}{2\pi}\int^{\infty}_{-\infty}(\xi^{2}+\eta)\left(l_{1}|\phi_{1}|^{2}+l_{2}|\phi_{2}|^{2}\right)d\xi+\rho\int^{L}_{0}|V_{t}|^{2}dx+\mu\int^{L}_{0}|P_{t}|^{2}dx,
\end{align*}
where $\eta_{1}$ is an arbitrary constant, and we will give its range of values at the end of this section.
\end{lemma}
\begin{proof}
To prove the inequality, we first need to find the derivative of $I_{1}$. That is
\begin{align*}
\frac{d}{dt}I_{1}=\int^{L}_{0}\left(\rho V_{tt}V+\rho|V_{t}|^{2}+\mu P_{tt}P+\mu|P_{t}|^{2}\right)dx.
\end{align*}
Using the equations \eqref{5.11} and \eqref{5.13}, we have $\rho V_{tt}=\alpha V_{xx}-\gamma\beta P_{xx}-\delta\theta_{x}$ and $\mu P_{tt}=\beta P_{xx}-\gamma\beta V_{xx}$. Then, we obtain
\begin{align}
\frac{d}{dt}I_{1}=\int^{L}_{0}\left(\left(\alpha V_{xx}-\gamma\beta P_{xx}\right)V-\delta\theta_{x}V+\left(\beta P_{xx}-\gamma\beta V_{xx}\right)P+\rho|V_{t}|^{2}+\mu|P_{t}|^{2}\right) dx.
\end{align}
Integrating by parts and using the boundary conditions \eqref{bc1}-\eqref{5.16}, we have
\begin{equation}\label{6.4}
\begin{aligned}
\frac{d}{dt}I_{1}=&\int^{L}_{0}\left[\rho|V_{t}|^{2}+\mu|P_{t}|^{2}-\alpha_{1}|V_{x}|^{2}-\beta|\gamma V_{x}-P_{x}|^{2}-\delta\theta_{x}V\right]dx\\
&-l_{1}\partial^{a,\eta}_{t}V(L,t)V(L,t)-l_{2}\partial^{a,\eta}_{t}P(L,t)P(L,t).
\end{aligned}
\end{equation}
Then using H\"{o}lder's inequality and Young's inequality $ab\leq\frac{a^2}{4\varepsilon _{i}}+ \varepsilon _{i}b^2$ with $a,b\in\mathbb{R}$ , $\varepsilon _{i}>0$, $i=1,2,3$ and combining \eqref{nuse2}, we can obtain
\begin{equation}\label{6.400}
\begin{aligned}
\left|-l_{1}\partial^{a,\eta}_{t}V(L,t)V(L,t)\right|&\leq \frac{1}{4\varepsilon_1}\left[l_{1}\partial^{a,\eta}_{t}V(L,t)\right]^2+\varepsilon_1|V(L,t)|^2\\
&=\frac{1}{4\varepsilon_1}\left[l_{1}\frac{\sin(a\pi)}{\pi}\int^{\infty}_{-\infty}\mu(\xi)\phi_{1}(\xi,t)d\xi\right]^2+\varepsilon_1\left[\int^L_0 V_x(x,t) dx\right]^2\\
&\leq \frac{M}{4\varepsilon_1}\frac{\sin(a\pi)}{\pi}\int^{\infty}_{-\infty}(\xi^{2}+\eta)\left(l_{1}|\phi_{1}|^{2}+l_{2}|\phi_{2}|^{2}\right)d\xi+\varepsilon_1 L\int^L_0 |V_x|^2dx.
\end{aligned}
\end{equation}
Similarly, we can get
\begin{equation}\label{6.401}
\begin{aligned}
\left|-l_{2}\partial^{a,\eta}_{t}P(L,t)P(L,t)\right|&\leq \frac{1}{4\varepsilon_2}\left[l_{1}\partial^{a,\eta}_{t}P(L,t)\right]^2+\varepsilon_2|P(L,t)|^2\\
&\leq \frac{M}{4\varepsilon_2}\frac{\sin(a\pi)}{\pi}\int^{\infty}_{-\infty}(\xi^{2}+\eta)\left(l_{1}|\phi_{1}|^{2}+l_{2}|\phi_{2}|^{2}\right)d\xi+\varepsilon_2 L\int^L_0 |P_x|^2dx.
\end{aligned}
\end{equation}
Now we notice that
\begin{align}\label{6.5}
\int^{L}_{0}|P_{x}|^{2}dx =\int^{L}_{0}\left|\gamma V_{x}-P_{x}-\gamma V_{x}\right|^{2}dx \leq 2\int^{L}_{0}\left|\gamma V_{x}-P_{x}\right|^{2}dx+2\gamma^{2}\int^{L}_{0}\left|V_{x}\right|^{2}dx.
\end{align}
Note that when we combine \eqref{6.400} with \eqref{6.401} and \eqref{6.5}, and use Young's inequality and Poincar\'{e}'s inequality into \eqref{6.4},  then \eqref{6.4} can be further rewritten as
\begin{align*}
\frac{d}{dt}I_{1}\leq& -\left(\alpha_{1}-\varepsilon_{1}L-2\gamma^{2}\varepsilon_{2}L-\varepsilon_{3}C_p\right)\int^{L}_{0}|V_{x}|^{2}dx-\left(\beta-2\varepsilon_{2}L\right)\int^{L}_{0}|\gamma V_{x}-P_{x}|^{2}dx\\
&+\frac{\delta^{2}}{4\varepsilon_{3}}\int^{L}_{0}|\theta_{x}|^{2}dx
+\left(\frac{M}{4\varepsilon_{1}}+\frac{M}{4\varepsilon_{2}}\right)\frac{\sin(a\pi)}{\pi}\int^{\infty}_{-\infty}(\xi^{2}+\eta)\left(l_{1}|\phi_{1}|^{2}+l_{2}|\phi_{2}|^{2}\right)d\xi\\
&+\rho\int^{L}_{0}|V_{t}|^{2}dx+\mu\int^{L}_{0}|P_{t}|^{2}dx.
\end{align*}
The above formula holds for any $\varepsilon _{i}>0$. By defining a new constant $\eta_1$ and choosing the appropriate $\varepsilon _{i}>0$ such that $\varepsilon _{1}L=\varepsilon _{2}L=\varepsilon _{3}C_p=\frac{\eta_{1}}{2}$, we arrive at
\begin{align*}
\frac{d}{dt}I_{1}\leq& -\left(\alpha_{1}-(1+\gamma^{2})\eta_{1}\right)\int^{L}_{0}|V_{x}|^{2}dx-\left(\beta-\eta_{1}\right)\int^{L}_{0}|\gamma V_{x}-P_{x}|^{2}dx
+\frac{\delta^2 C_p}{2\eta_{1}}\int^{L}_{0}|\theta_{x}|^{2}dx\\
&+\frac{ML}{\eta_{1}}\frac{\sin(a\pi)}{2\pi}\int^{\infty}_{-\infty}(\xi^{2}+\eta)\left(l_{1}|\phi_{1}|^{2}+l_{2}|\phi_{2}|^{2}\right)d\xi+\rho\int^{L}_{0}|V_{t}|^{2}dx+\mu\int^{L}_{0}|P_{t}|^{2}dx.
\end{align*}
The above formula holds for any $\eta_1>0$. Thus, the proof of the lemma is completed.
\end{proof}
\begin{lemma}\label{lem6.2}
Let
\begin{align*}
I_{2}=\rho c \int^{L}_{0}\theta(x,t)\int^{L}_{x} V_{t}(y,t)dy dx,
\end{align*}
then
\begin{align*}
\frac{d}{dt}I_{2}\leq& \eta_{2} \int^{L}_{0}|V_{x}|^{2}dx+\eta_{2}\int^{L}_{0}|\gamma V_{x}-P_{x}|^{2}dx-\left(\delta\rho-\frac{1}{2\gamma^{2}\beta^{2}}\eta_{2}\right)\rho\int^{L}_{0}|V_{t}|^{2}dx\\
&+\left(C_{p}c\delta+\frac{\gamma^{2}\beta^{2}\rho^{2}k^{2}}{2\eta_{2}}+\frac{c^{2}C_{p}\alpha^{2}_{1}}{2\eta_{2}}+\frac{c^2 L C_p}{2 \eta_2}\right)\int^{L}_{0}|\theta_{x}|^{2}dx\\
&+\eta_2 M \frac{\sin(a\pi)}{2\pi}\int^{\infty}_{-\infty}(\xi^{2}+\eta)\left(l_{1}|\phi_{1}|^{2}+l_{2}|\phi_{2}|^{2}\right)d\xi,
\end{align*}
where $\eta_{2}$ is an arbitrary constant, $C_{p}$ is the Poincar\'{e} constant, and we will give its range of values at the end of this section.
\end{lemma}
\begin{proof}
Finding the derivative of $I_{2}$, we have
\begin{align*}
\frac{d}{dt}I_{2}=\rho c \int^{L}_{0}\theta_{t}(x,t)\int^{L}_{x} V_{t}(y,t)dy dx+\rho c \int^{L}_{0}\theta(x,t)\int^{L}_{x} V_{tt}(y,t)dy dx.
\end{align*}
 Using the fact that $\rho V_{tt}=\alpha V_{xx}-\gamma\beta P_{xx}-\delta\theta_{x}$, $c\theta_{t}=k_{xx}-\delta V_{xt}$, we have
\begin{align*}
\frac{d}{dt}I_{2}=\rho\int^{L}_{0}\left(k\theta_{xx}-\delta V_{xt}\right)\int^{L}_{x} V_{t} dydx+c\int^{L}_{0}\theta\int^{L}_{x}\left(\alpha V_{xx}-\gamma\beta P_{xx}-\delta\theta_{x}\right)dydx.
\end{align*}
With the help of boundary conditions \eqref{bc1}-\eqref{5.16}, we have
\begin{align*}
\frac{d}{dt}I_{2}=&\rho\int^{L}_{0}k\theta_{x}V_{t}dx-\delta\rho\int^{L}_{0}|V_{t}|^{2}dx-c\int^{L}_{0}\theta\left(\alpha V_{x}-\gamma\beta P_{x}\right)dx+c\delta\int^{L}_{0}|\theta|^{2}dx\\
&+c\left[(\alpha V_x-\gamma\beta P_x)(L,t)\right]\int^{L}_{0}\theta(x,t)dx.
\end{align*}
Using Poincar\'{e}'s inequality and Young's inequality $ab\leq\frac{a}{4\varepsilon _{i}}+ \varepsilon _{i}b$ with $a,b\in\mathbb{R}$, $\varepsilon _{i}>0$, $i=2,3,4$ and using \eqref{nuse2}, we get
\begin{align*}
\frac{d}{dt}I_{2}\leq& 2\varepsilon_{3}\alpha^{2}_{1} \int^{L}_{0}|V_{x}|^{2}dx+2\varepsilon_{2}\gamma^{2}\beta^{2}\int^{L}_{0}|\gamma V_{x}-P_{x}|^{2}dx-\left(\delta\rho-\varepsilon_{2}\right)\int^{L}_{0}|V_{t}|^{2}dx\\
&+\left(C_{p}c\delta+\frac{\rho^{2}k^{2}}{4\varepsilon_{2}}+\frac{c^{2}C_{p}}{4\varepsilon_{3}}+\frac{c^2 LC_p}{4\varepsilon_4}\right)\int^{L}_{0}|\theta_{x}|^{2}dx\\
&+\varepsilon_4M\frac{\sin(a\pi)}{\pi}\int^{\infty}_{-\infty}(\xi^{2}+\eta)\left(l_{1}|\phi_{1}|^{2}+l_{2}|\phi_{2}|^{2}\right)d\xi.
\end{align*}
The above formula holds for any $\varepsilon _{i}>0$. Let us define a new constant $\eta_{2}$ and choose $\varepsilon _{i}>0$ such that $\varepsilon _{2}=\frac{\eta_{2}}{2\gamma^{2}\beta^{2}}$, $\varepsilon _{3}=\frac{\eta_{2}}{2\alpha^{2}_{1}}$, $\varepsilon_4=\frac{\eta_2}{2}$. Then, we get
\begin{align*}
\frac{d}{dt}I_{2}\leq& \eta_{2} \int^{L}_{0}|V_{x}|^{2}dx+\eta_{2}\int^{L}_{0}|\gamma V_{x}-P_{x}|^{2}dx-\left(\delta\rho-\frac{1}{2\gamma^{2}\beta^{2}}\eta_{2}\right)\rho\int^{L}_{0}|V_{t}|^{2}dx\\
&+\left(C_{p}c\delta+\frac{\gamma^{2}\beta^{2}\rho^{2}k^{2}}{2\eta_{2}}+\frac{c^{2}C_{p}\alpha^{2}_{1}}{2\eta_{2}}+\frac{c^2 L C_p}{2 \eta_2}\right)\int^{L}_{0}|\theta_{x}|^{2}dx\\
&+\eta_2 M \frac{\sin(a\pi)}{2\pi}\int^{\infty}_{-\infty}(\xi^{2}+\eta)\left(l_{1}|\phi_{1}|^{2}+l_{2}|\phi_{2}|^{2}\right)d\xi.
\end{align*}
Now, we obtain that the lemma is accurate.
\end{proof}
\begin{lemma}\label{lem6.3}
Let
\begin{align*}
I_{3}=\rho\int^{L}_{0}  V_{t}V dx+\gamma\mu \int^{L}_{0}P_{t}V dx,
\end{align*}
then
\begin{align*}
\frac{d}{dt}I_{3}\leq& -\left(\alpha_{1}-\frac{\alpha_{1}}{4\eta_{3}}\right)\int^{L}_{0}|V_{x}|^{2}dx+(\rho+\eta_{3})\int^{L}_{0}|V_{t}|^{2}dx+\frac{\gamma^{2}\mu^{2}}{4\eta_{3}}\int^{L}_{0}|P_{t}|^{2}dx\\
&+\frac{4ML\eta_{3}}{\alpha_{1}}\frac{\sin(a\pi)}{2\pi}\int^{\infty}_{-\infty}(\xi^{2}+\eta)\left(l_{1}|\phi_{1}|^{2}+l_{2}|\phi_{2}|^{2}\right)d\xi+\frac{4C_{p}\delta^{2}\eta_{3}}{\alpha_{1}}\int^{L}_{0}|\theta_{x}|^{2}dx,
\end{align*}
where $\eta_{3}$ is an arbitrary constant, $C_{p}$ is the Poincar\'{e} constant, and we will give its range of values at the end of this section.
\end{lemma}
\begin{proof}
To prove the inequality, we first need to find the derivative of $I_3$. That is
\begin{align*}
\frac{d}{dt}I_{3}=\int^{L}_{0}\left(\rho V_{tt}V+\rho|V_{t}|^{2}+\gamma\mu P_{tt}V+\gamma\mu P_{t}V_{t} \right)dx.
\end{align*}
Here we also make use of the fact that $\rho V_{tt}=\alpha V_{xx}-\gamma\beta P_{xx}-\delta\theta_{x}$ and $\alpha=\alpha_{1}+\gamma^{2}\beta$ in equation \eqref{5.11}. Then the above equality can be further rewritten as
\begin{align*}
\frac{d}{dt}I_{3}=&\rho\int^{L}_{0}|V_{t}|^{2}dx-\alpha_{1}\int^{L}_{0}|V_{x}|^{2}dx +\alpha_{1}V_x(L,t)V(L,t)+\int^L_0 \delta\theta V_{x}dx\\
 &+\int^{L}_{0}\gamma\beta(\gamma V-P)_{xx}Vdx
+\int^{L}_{0}\gamma\mu P_{tt}Vdx+\int^L_0 \gamma\mu P_{t}V_{t}dx.
\end{align*}
Using the equation \eqref{5.13}, we have $\mu P_{tt}=\beta P_{xx}-\gamma\beta V_{xx}$. Then we obtain
\begin{align*}
\frac{d}{dt}I_{3}=&\rho\int^{L}_{0}|V_{t}|^{2}dx-\alpha_{1}\int^{L}_{0}|V_{x}|^{2}dx +\alpha_{1}V_x(L,t)V(L,t)+\int^{L}_{0}\delta\theta V_{x}dx +\int^{L}_{0}\gamma\mu P_{t}V_{t}dx.
\end{align*}
Using Poincar\'{e}'s inequality and Young's inequality $ab\leq\frac{a}{4\varepsilon _{i}}+ \varepsilon _{i}b$ with $a,b\in\mathbb{R}$, $\varepsilon _{i}>0$, $i=3,4,5$, we have
\begin{align*}
\frac{d}{dt}I_{3}\leq& -\left(\alpha_{1}-\varepsilon_{3}L-\varepsilon_{4}\right)\int^{L}_{0}|V_{x}|^{2}dx+(\rho+\varepsilon_{5})\int^{L}_{0}|V_{t}|^{2}dx+\frac{\gamma^{2}\mu^{2}}{4\varepsilon_{5}}\int^{L}_{0}|P_{t}|^{2}dx\\
&+\frac{M}{2\varepsilon_{3}}\frac{\sin(a\pi)}{2\pi}\int^{\infty}_{-\infty}(\xi^{2}+\eta)\left(l_{1}|\phi_{1}|^{2}+l_{2}|\phi_{2}|^{2}\right)d\xi+\frac{C_{p}\delta^{2}}{4\varepsilon_{4}}\int^{L}_{0}|\theta_{x}|^{2}dx.
\end{align*}
The above formula holds for any $\varepsilon _{i}>0$. By defining a new constant $\eta_{3}$ and choosing the appropriate $\varepsilon _{i}>0$ such that $\varepsilon _{3}L=\frac{\alpha_{1}}{8\eta_{3}}$, $\varepsilon _{4}=\frac{\alpha_{1}}{8\eta_{3}}$, $\varepsilon_{5}=\eta_{3}$, we arrive at
\begin{align*}
\frac{d}{dt}I_{3}\leq& -\left(\alpha_{1}-\frac{\alpha_{1}}{4\eta_{3}}\right)\int^{L}_{0}|V_{x}|^{2}dx+(\rho+\eta_{3})\int^{L}_{0}|V_{t}|^{2}dx+\frac{\gamma^{2}\mu^{2}}{4\eta_{3}}\int^{L}_{0}|P_{t}|^{2}dx\\
&+\frac{4ML\eta_{3}}{\alpha_{1}}\frac{\sin(a\pi)}{2\pi}\int^{\infty}_{-\infty}(\xi^{2}+\eta)\left(l_{1}|\phi_{1}|^{2}+l_{2}|\phi_{2}|^{2}\right)d\xi+\frac{4C_{p}\delta^{2}\eta_{3}}{\alpha_{1}}\int^{L}_{0}|\theta_{x}|^{2}dx.
\end{align*}
The above formula holds for any $\eta_3>0$. Thus, the proof of the lemma is completed.
\end{proof}

\begin{lemma}\label{lem6.4}
Let
\begin{align*}
I_{4}=\rho\int^{L}_{0} V_{t}(\gamma V-P)dx+\gamma\mu \int^{L}_{0}P_{t}(\gamma V-P)dx,
\end{align*}
then
\begin{equation}\label{6.40}
\begin{aligned}
\frac{d}{dt}I_{4}\leq& \frac{\alpha_{1}^{2}}{\eta_{4}}\int^{L}_{0}|V_{x}|^{2}dx+\frac{3\eta_{4}}{4}\int^{L}_{0}|\gamma V_{x}-P_{x}|^{2}dx
+\frac{\delta^{2}C_{p}}{\eta_{4}}\int^{L}_{0}|\theta_{x}|^{2}dx\\
&+\frac{\alpha_{1}^{2}M}{\eta_{4}}\frac{\sin(a\pi)}{2\pi}\int^{\infty}_{-\infty}(\xi^{2}+\eta)\left(l_{1}|\phi_{1}|^{2}+l_{2}|\phi_{2}|^{2}\right)d\xi\\
&+\left[\gamma\rho+\frac{1}{\eta_{4}}\left(\rho^{2}+\gamma^{4}\mu^{2}\right)\right]\int^{L}_{0}|V_{t}|^{2}dx
-\left(\gamma\mu-\frac{1}{2}\eta_{4}\right)\int^{L}_{0}|P_{t}|^{2}dx.
\end{aligned}
\end{equation}
where $\eta_{4}$ is an arbitrary constant, we will give its range of values at the end of this section. $C_{p}$ is the Poincare constant.
\end{lemma}
\begin{proof}
Finding the derivative of $I_{4}$, we have
\begin{align*}
\frac{d}{dt}I_{4}=\int^{L}_{0}\bigg[\rho V_{tt}(\gamma V-P)+\rho V_{t}(\gamma V-P)_{t}+\gamma\mu P_{tt}(\gamma V-P)+\gamma\mu P_{t}(\gamma V-P)_{t}\bigg]dx.
\end{align*}
Using the equation \eqref{5.11}, we have $\rho V_{tt}=\alpha_{1} V_{xx}+ \gamma^{2}\beta V_{xx}-\gamma\beta P_{xx}-\delta\theta_{x}$. And integrating by parts, we have
\begin{align*}
\frac{d}{dt}I_{4}=&\int^{L}_{0}\bigg[\rho V_{t}(\gamma V-P)_t+\gamma\mu P_{t}(\gamma V-P)_{t}\bigg]dx-\int^{L}_{0}\bigg[\alpha_{1}V_{x}(\gamma V-P)_{x}-\delta\theta(\gamma V-P)_{x}\bigg]dx\\
&+\int^{L}_{0}\gamma(\gamma\beta V_{xx}-\beta P_{xx})(\gamma V-P)dx+\gamma\int^{L}_{0}\mu P_{tt}(\gamma V-P)dx+\alpha_{1} V_{x}(L,t)(\gamma V-P)(L,t).
\end{align*}
Using the equation \eqref{5.13}, we have $\mu P_{tt}=\beta P_{xx}-\gamma\beta V_{xx}$. Then the third and fourth terms in the above equality become
\begin{align*}
\int^{L}_{0}\gamma(\gamma\beta V_{xx}-\beta P_{xx})(\gamma V-P)dx+\gamma\int^{L}_{0}\mu P_{tt}(\gamma V-P)dx=0.
\end{align*}
After sorting out the above equality, we can get
\begin{align*}
\frac{d}{dt}I_{4}=&\gamma\rho\int^{L}_{0}|V_{t}|^{2}dx-\gamma\mu\int^{L}_{0}|P_{t}|^{2}dx-\int^{L}_{0}\rho V_{t}P_{t}dx+\int^{L}_{0}\gamma^{2}\mu P_{t}V_{t}dx\\
&-\alpha_{1}\int^{L}_{0}V_{x}(\gamma V-P)_{x}dx+\delta\int^{L}_{0}\theta(\gamma V-p)_{x}dx+\alpha_{1}V_{x}(L,t)(\gamma V-P)(L,t).
\end{align*}
Thanks to Poincar\'{e}'s inequality and Young's inequality $ab\leq\frac{a}{4\varepsilon _{i}}+ \varepsilon _{i}b$ with $a,b\in\mathbb{R}$ , $\varepsilon _{i}>0$, $i=4,5,6,7,8$, we obtain
\begin{align*}
\frac{d}{dt}I_{4}\leq& \frac{\alpha_{1}^{2}}{4\varepsilon_{6}}\int^{L}_{0}|V_{x}|^{2}dx+(\varepsilon_{6}+\varepsilon_{7}+\varepsilon_{8})\int^{L}_{0}|\gamma V_{x}-P_{x}|^{2}dx
+\frac{\delta^{2}C_{p}}{4\varepsilon_{7}}\int^{L}_{0}|\theta_{x}|^{2}dx\\
&+\frac{\alpha_{1}^{2}M}{\eta_{4}}\frac{\sin(a\pi)}{2\pi}\int^{\infty}_{-\infty}(\xi^{2}+\eta)\left(l_{1}|\phi_{1}|^{2}+l_{2}|\phi_{2}|^{2}\right)d\xi\\
&+\left(\gamma\rho+\frac{\rho^{2}}{4\varepsilon_{4}}\frac{\gamma^{4}\mu^{2}}{4\varepsilon_{5}}\right)\int^{L}_{0}|V_{t}|^{2}dx
-\left(\gamma\mu-\varepsilon_{4}-\varepsilon_{5} \right)\int^{L}_{0}|P_{t}|^{2}dx.
\end{align*}
The above formula holds for any $\varepsilon _{i}>0$. By defining a new constant $\eta_{4}$ and choosing $\varepsilon _i=\frac{\eta_{4}}{4}$, $i={4,5,6,7,8}$, we obtain the inequality \eqref{6.40} is right.
\end{proof}

Next, we begin to prove the main result of this section. That is, the piezoelectric beam equation system with Fourier's law is exponentially stable. Before that, we give a definition of functional Lyapunov
\begin{align*}
\mathcal{L}(t)=N E(t)+N_{1} I_{1}+N_{2}I_{2}+N_{3} I_{3}+N_{4}I_{4},
\end{align*}
where $E(t)$ is the total energy of the system with Fourier's law heat and the functional $I_{1}$, $I_{2}$, $I_{3}$ and $I_{4}$ are given in Lemma \ref{lem6.1}, Lemma \ref{lem6.2}, Lemma \ref{lem6.3} and Lemma \ref{lem6.4} respectively, and $N$, $N_{i}$ for $i = 1, 2, 3, 4$ are some constants for which we will give a specific range later.
\begin{lemma}\label{lem6.5}
Let $N$ be a large enough constant, then there are positive constants $m_{1}$ and $m_{2}$ such that the functional
\begin{align*}
\mathcal{L}(t)=N E(t)+N_{1} I_{1}+N_{2}I_{2}+N_{3} I_{3}+N_{4}I_{4},
\end{align*}
satisfies
\begin{align*}
m_{1}E(t)\leq \mathcal{L}(t)\leq m_{2}E(t).
\end{align*}
\end{lemma}
\begin{proof}
Obviously, from the definition of the functional $\mathcal{L}(t)$, we can show
\begin{align*}
|\mathcal{L}(t)-N E(t)|=|N_{1} I_{1}+N_{2}I_{2}+N_{3} I_{3}+N_{4}I_{4}|.
\end{align*}
By applying Young's inequality and Poincar\'{e}'s inequality, we conclude that there is a constant $C > 0$ such that
\begin{align*}
|\mathcal{L}(t)-N E(t)|\leq C E(t).
\end{align*}
That is,
\begin{align*}
(N-C)E(t)\leq \mathcal{L}(t)\leq (N+C)E(t).
\end{align*}
 Thus, the proof of the lemma is completed.
\end{proof}
\begin{theorem}
Let $(V, V_{t}, \phi_{1}, P, P_{t}, \phi_{2}, \theta)$ be any solution of system \eqref{5.11}-\eqref{5.17}, then there are constants $M > 0$ and $\omega > 0$
independent of the initial conditions such that
\begin{align*}
E(t)\leq E(0)e^{-\omega t}.
\end{align*}
\end{theorem}
\begin{proof}
By employing Lemma \ref{lem6.1}- Lemma \ref{lem6.4}, we arrive at
\begin{align}
\frac{d}{dt}\mathcal{L}(t)=& N \frac{d}{dt}E(t)+N_{1}\frac{d}{dt} I_{1}+N_{2}\frac{d}{dt}I_{2}+N_{3}\frac{d}{dt} I_{3}+N_{4}\frac{d}{dt}I_{4}\notag\\
\leq &-\left(N-N_{1}\frac{ML}{\eta_{1}}-N_2\eta_2 M-N_{3}\frac{4ML\eta_{3}}{\alpha_{1}}-N_{4}\frac{\alpha_{1}^{2}M}{\eta_{4}}\right)
\frac{\sin(a\pi)}{2\pi}\int^{\infty}_{-\infty}(\xi^{2}+\eta)\left(l_{1}|\phi_{1}|^{2}+l_{2}|\phi_{2}|^{2}\right)d\xi\notag\\
&-\left(Nk-N_{1}\frac{\delta^{2}C_p}{2\eta_{1}}-N_{2}C_{\eta_{2}}-N_{3}\frac{4C_{p}\delta^{2}\eta_{3}}{\alpha_{1}}-N_{4}\frac{C_{p}\delta^{2}}{\eta_{4}}\right)\int^{L}_{0}|\theta_{x}|^{2}dx\notag\\
&-\left(N_{2}\delta\rho-N_{2}\frac{\eta_{2}}{2\gamma^{2}\beta^{2}}-N_{1}\rho-N_{3}\rho-N_{3}\eta_{3}-N_{4}\gamma\rho-N_{4}\frac{\rho^{2}+\gamma^{4}\mu^{2}}{\eta_{4}}\right)\int^{L}_{0}|V_{t}|^{2}dx\notag\\
&-\left(N_{4}\gamma\mu-N_{4}\frac{\eta_{4}}{2}-N_{1}\mu-N_{3}\frac{\gamma^{2}\mu^{2}}{4\eta_{4}}\right)\int^{L}_{0}|P_{t}|^{2}dx\notag\\
&-\left(N_{1}\alpha_{1}-N_{1}(1+\gamma^{2})\eta_{1}-N_{2}\eta_{2}+N_{3}\alpha_{1}-N_{3}\frac{\alpha_{1}}{4\eta_{3}}-N_{4}\frac{\alpha^{2}_{1}}{\eta_{4}}\right)\int^{L}_{0}|V_{x}|^{2}dx\notag\\
&-\left(N_{1}\beta-N_{1}\eta_{1}-N_{2}\eta_{2}-N_{4}\frac{3\eta_{4}}{4}\right)\int^{L}_{0}|\gamma V_{x}-P_{x}|^{2}dx,\label{ly0}
\end{align}
where $C_{\eta_{2}}=C_{p}c\delta+\frac{\gamma^{2}\beta^{2}\rho^{2}k^{2}}{2\eta_{2}}+\frac{c^{2}C_{p}\alpha^{2}_{1}}{2\eta_{2}}+\frac{c^2 C_p L}{2\eta_2}$. In order to make the terms in \eqref{ly0} the same as those in the energy $E(t)$, we focus on the first term at the right hand side of \eqref{ly0}. By using the fact that $\eta,l_1,l_2>0, a\in(0,1)$, we obtain
\begin{align*}
-\frac{\sin(a\pi)}{2\pi}\int^{\infty}_{-\infty}(\xi^{2}+\eta)\left(l_{1}|\phi_{1}|^{2}+l_{2}|\phi_{2}|^{2}\right)d\xi\leq
-\eta \frac{\sin(a\pi)}{2\pi}\int^{\infty}_{-\infty}\left(l_{1}|\phi_{1}|^{2}+l_{2}|\phi_{2}|^{2}\right)d\xi.
\end{align*}
Then, from the above discussions, \eqref{ly0} can be rewritten as follows
\begin{align}
\frac{d}{dt}\mathcal{L}(t)
\leq &-\left(N-N_{1}\frac{ML}{\eta_{1}}-N_2 \eta_2 M-N_{3}\frac{4ML\eta_{3}}{\alpha_{1}}-N_{4}\frac{\alpha_{1}^{2}M}{\eta_{4}}\right)
\eta \frac{\sin(a\pi)}{2\pi}\int^{\infty}_{-\infty}\left(l_{1}|\phi_{1}|^{2}+l_{2}|\phi_{2}|^{2}\right)d\xi\notag\\
&-\left(Nk-N_{1}\frac{\delta^{2}C_p}{2\eta_{1}}-N_{2}C_{\eta_{2}}-N_{3}\frac{4C_{p}\delta^{2}\eta_{3}}{\alpha_{1}}-N_{4}\frac{C_{p}\delta^{2}}{\eta_{4}}\right)\int^{L}_{0}|\theta_{x}|^{2}dx\notag\\
&-\left(N_{2}\delta\rho-N_{2}\frac{\eta_{2}}{2\gamma^{2}\beta^{2}}-N_{1}\rho-N_{3}\rho-N_{3}\eta_{3}-N_{4}\gamma\rho-N_{4}\frac{\rho^{2}+\gamma^{4}\mu^{2}}{\eta_{4}}\right)\int^{L}_{0}|V_{t}|^{2}dx\notag\\
&-\left(N_{4}\gamma\mu-N_{4}\frac{\eta_{4}}{2}-N_{1}\mu-N_{3}\frac{\gamma^{2}\mu^{2}}{4\eta_{4}}\right)\int^{L}_{0}|P_{t}|^{2}dx\notag\\
&-\left(N_{1}\alpha_{1}-N_{1}(1+\gamma^{2})\eta_{1}-N_{2}\eta_{2}+N_{3}\alpha_{1}-N_{3}\frac{\alpha_{1}}{4\eta_{3}}-N_{4}\frac{\alpha^{2}_{1}}{\eta_{4}}\right)\int^{L}_{0}|V_{x}|^{2}dx\notag\\
&-\left(N_{1}\beta-N_{1}\eta_{1}-N_{2}\eta_{2}-N_{4}\frac{3\eta_{4}}{4}\right)\int^{L}_{0}|\gamma V_{x}-P_{x}|^{2}dx.\label{ly1}
\end{align}
Choose the appropriate constants as
\begin{align*}
\frac{N_{3}}{2}=\eta_{3},\quad 3N_{4}\eta_{4}=4\beta,\quad N_{1}\eta_{1}=\beta, \quad N_{2}\eta_{2}=\beta.
\end{align*}
And we simply denote $\frac{3 \alpha^{2}_{1}N_{4}}{4\beta}$ as $C_{N_{4}}$, then \eqref{ly1} can be rewritten as
\begin{align*}
\frac{d}{dt}\mathcal{L}(t)
\leq &-\left(N-N_{1}\frac{ML}{\eta_{1}}-N_2 \eta_2 M-N_{3}\frac{4ML\eta_{3}}{\alpha_{1}}-N_{4}\frac{\alpha_{1}^{2}M}{\eta_{4}}\right)
\eta \frac{\sin(a\pi)}{2\pi}\int^{\infty}_{-\infty}\left(l_{1}|\phi_{1}|^{2}+l_{2}|\phi_{2}|^{2}\right)d\xi\\
&-\left(Nk-N_{1}\frac{\delta^{2}C_p}{2\eta_{1}}-N_{2}C_{\eta_{2}}-N_{3}\frac{4C_{p}\delta^{2}\eta_{3}}{\alpha_{1}}-N_{4}\frac{C_{p}\delta^{2}}{\eta_{4}}\right)\int^{L}_{0}|\theta_{x}|^{2}dx\\
&-\left(N_{2}\delta\rho-\frac{1}{2\gamma^{2}\beta}-N_{1}\rho-N_{3}\rho-N_{3}\eta_{3}-N_{4}\gamma\rho-\frac{3N^{2}_{4}}{4\beta}\left(\rho^{2}+\gamma^{4}\mu^{2}\right)\right)\int^{L}_{0}|V_{t}|^{2}dx\\
&-\left(N_{4}\gamma\mu-\frac{2\beta}{3}-N_{1}\mu-\frac{\gamma^{2}\mu^{2}}{8}\right)\int^{L}_{0}|P_{t}|^{2}dx\\
&-\left(N_{1}\alpha_{1}-(2+\gamma^{2})\beta+N_{3}\alpha_{1}-\frac{\alpha_{1}}{2}-C_{N_{4}}N_{4}\right)\int^{L}_{0}|V_{x}|^{2}dx\\
&-\left(N_{1}\beta-3\beta\right)\int^{L}_{0}|\gamma V_{x}-P_{x}|^{2}dx.
\end{align*}
Let
\begin{align*}
\lambda_{1}:&=N-N_{1}\frac{ML}{\eta_{1}}-N_2 \eta_2 M-N_{3}\frac{4ML\eta_{3}}{\alpha_{1}}-N_{4}\frac{\alpha_{1}^{2}M}{\eta_{4}};\\
\lambda_{2}:&=Nk-N_{1}\frac{\delta^{2}C_p}{2\eta_{1}}-N_{2}C_{\eta_{2}}-N_{3}\frac{4C_{p}\delta^{2}\eta_{3}}{\alpha_{1}}-N_{4}\frac{C_{p}\delta^{2}}{\eta_{4}};\\
\lambda_{3}:&=N_{2}\delta\rho-\frac{1}{2\gamma^{2}\beta}-N_{1}\rho-N_{3}\rho-N_{3}\eta_{3}-N_{4}\gamma\rho-\frac{3N^{2}_{4}}{4\beta}\left(\rho^{2}+\gamma^{4}\mu^{2}\right);\\
\lambda_{4}:&=N_{4}\gamma\mu-\frac{2\beta}{3}-N_{1}\mu-\frac{\gamma^{2}\mu^{2}}{8};\\
\lambda_{5}:&=N_{1}\alpha_{1}-(2+\gamma^{2})\beta+N_{3}\alpha_{1}-\frac{\alpha_{1}}{2}-C_{N_{4}}N_{4};\\
\lambda_{6}:&=N_{1}\beta-3\beta.
\end{align*}
Now, we will choose suitable constants $N$ and $N_{i}\ (i=1,2,3,4)$ so that $\lambda_{i}\ (i=1,...,6)$ are all positive. From $\lambda_{6}>0$, we can obtain that $N_{1}$ should satisfy
\begin{align}\label{ly2}
N_{1}>3.
\end{align}
In order to achieve $\lambda_{4}>0$, we use the result in \eqref{ly2}, and obtain that $N_{4}$ satisfies
\begin{align}\label{ly3}
N_{4}>\frac{2\beta}{3\gamma\mu}+\mu N_{1}+\frac{\gamma\mu}{8}.
\end{align}
By using similar techniques, we have that $N_{3}$ and $N_{2}$ satisfies
\begin{align*}
N_{3}&>\frac{(2+\gamma^{2})\beta}{\alpha_{1}}+\frac{1}{2}+\frac{C_{N_{4}}N_{4}}{\alpha_{1}}-N_{1},\\
N_{2}&>\frac{1}{2\gamma^{2}\beta\delta\rho}+\frac{ N_{1}}{\delta}+\frac{ N_{3}}{\delta}+\frac{ N_{3}^{2}}{2\delta\rho}+\frac{\gamma N_{4}}{\delta}+\frac{3 N_{4}^{2}}{4\beta\delta\rho}\left(\rho^{2}+\gamma^{4}\mu^{2}\right).
\end{align*}
Once $N_i, i={1,2,3,4}$ are fixed, we choose $N$ sufficiently large, which can ensured that $\lambda_{1},\lambda_{2}>0$. Then, we can conclude that there exists a constant $N_{0}:=2\min\{{\lambda_{1}\eta}, \frac{\lambda_{2}}{c}, \frac{\lambda_{3}}{\rho}, \frac{\lambda_{4}}{\mu}, \frac{\lambda_{5}}{\alpha_{1}}, \frac{\lambda_{6}}{\beta}\}$ such that
\begin{align*}
\frac{d}{dt}\mathcal{L}(t)\leq -N_{0}E(t), \quad t\geq 0.
\end{align*}
Thanks to Lemma  \ref{lem6.5}, it is straightforward to verify that there exist constants $M > 0$ and $\omega > 0$ independent of the initial conditions such that
\begin{align*}
E(t)\leq E(0)e^{-\omega t}.
\end{align*}
 Thus, the proof of the theorem is completed.
\end{proof}

\subsection*{Acknowledgments}

This work was supported by the National Natural Science Foundation of China [grant number 11771216], the Key Research and Development Program of Jiangsu Province (Social Development) [grant number BE2019725], the Qing Lan Project of Jiangsu Province and
Postgraduate Research and Practice Innovation Program of Jiangsu Province (Grant No. KYCX20\_0945).
Special thanks go to Prof. Mohammad Akil for his valuable suggestions on Theorem \ref{thm4.5}.

\subsection*{Data Availability}
The data that support the findings of this study are available within the article.

\end{document}